\documentclass[a4paper,12pt]{article}
\usepackage[top=2.5cm,bottom=2.5cm,left=2.5cm,right=2.5cm]{geometry}
\usepackage{cite, amsmath, amssymb, cases, mathrsfs,mathtools}
\usepackage{extarrows}
\numberwithin{equation}{section}
\usepackage{amssymb}
\usepackage{graphicx}
\usepackage{array}
\newenvironment{proof}{{\noindent\textbf{\bf {Proof.}}}}{{\hfill 
$\square$\par}{\vspace{0.5\baselineskip}}}
\usepackage{latexsym}
\usepackage{graphicx,booktabs,multirow}
\usepackage{tikz}
\usetikzlibrary{arrows.meta, calc, fit, positioning, shapes}
\usetikzlibrary{decorations.pathreplacing}
\usetikzlibrary{intersections}
\usetikzlibrary{matrix, arrows.meta, positioning}
\usetikzlibrary{positioning, arrows.meta}
\usepackage{enumerate}
\usepackage{graphicx,booktabs,multirow}
\usepackage{appendix}
\usepackage{ntheorem}
\theoremseparator{.}
\newtheorem{theorem}{Theorem}[section]
\newtheorem{proposition}[theorem]{\rm\bfseries Proposition}
\newtheorem{lemma}[theorem]{Lemma}

\newtheorem{problem}{Problem}[section]
\newtheorem{claim}{Claim}

\newtheorem{definition}[theorem]{Definition}
\newtheorem{construction}{Construction}
\newtheorem{observation}[theorem]{Observation}
\usepackage[section]{placeins}
\def\NAT@def@citea{\def\@citea{\NAT@separator}}
\allowdisplaybreaks

\begin{document}
\vspace*{10mm}

\noindent
{\Large \bf Oriented diameter of graphs with diameter $4$ and given edge girth}

\vspace*{7mm}

\noindent
{\large \bf  Jifu Lin$^{1,2}$, Lihua You$^{2,*}$}
\noindent

\vspace{7mm}

\noindent
$^1$ Department of Mathematics, East China Normal University, Shanghai 200241, China, e-mail: {\tt jifulin01@163.com}.\\[2mm]
$^2$ School of Mathematical Sciences, South China Normal University, Guangzhou 510631, China,
e-mail: {\tt ylhua@scnu.edu.cn}.\\[2mm]
$^*$ Corresponding author
\noindent

\vspace{7mm}

\noindent
{\bf Abstract} \
\noindent
Let $f(d)$ be the smallest value for which every bridgeless graph
$G$ with diameter $d$ admits a strong orientation $\overrightarrow{G}$ such that the diameter of $\overrightarrow{G}$ is at most $f(d)$. Chvátal and Thomassen (JCT-B, 1978) obtained general bounds for $f(d)$ and proved that $f(2)=6$. Kwok et al. (JCT-B, 2010) proved that $9\leq f(3)\leq 11$. Wang and Chen (JCT-B, 2022) determined $f(3)=9$. Babu et al. (DAM, 2021) showed $f(4)\leq 21$.

In this paper, we introduce a new approach to studying $f(d)$ via the edge girth of a bridgeless graph $G$, denoted by $g^*(G)=\max\{l_G(e)\mid e\in E(G)\}$, where $l_G(e)$ is the length of the shortest cycle containing $e$ in $G$. Then we define $F(d,g^*)=\max\{\overrightarrow{{diam}}(G)\mid G\text{ is bridgeless},d(G)=d,g^*(G)=g^*\}$, and show $f(d)=\max\{F(d,g^*)\mid 2\leq g^*\leq 2d+1\}$. As the main result of this paper, we establish $F(4,2)=4$, $F(4,9)=12$, $F(4,3)\le 12$, and $F(4,g^*)\le 13$ for $g^*\in\{6,7,8\}$, and we propose two open problems for further research.
 \\[2mm]


\baselineskip=0.30in


\section{Introduction}

\hspace{1.5em}In this paper, we only consider a finite undirected connected graph $G=(V(G),E(G))$ without loops. For $S\subseteq V(G)$, we use $G[S]$ to denote the subgraph of $G$ induced by $S$. 
For any $u,v \in V(G)$, the distance $d(u,v)$ is the length of the shortest $(u,v)$-path in $G$, and the \emph{diameter} of $G$ is defined as $d(G)=\max\{d(u,v)|u,v\in V(G)\}$. A \emph{bridge} in a graph $ G $ is an edge whose deletion results in a disconnected graph. A graph $ G $ is called \emph{bridgeless} if it contains no bridges.

An \emph{orientation} $\overrightarrow{G}$ of $G$ is a digraph obtained from $G$ by assigning a direction to each edge, and $\overrightarrow{G}$ is \emph{strong} if and only if there exists a directed path between any two vertices in $\overrightarrow{G}$. The \emph{directed distance} $\partial(u,v)$ is the length of the shortest directed path from $u$ to $v$ in $\overrightarrow{G}$. If $\overrightarrow{G}$ is strong, then the diameter of $\overrightarrow{G}$ is defined as $d(\overrightarrow{G})=\max\{\partial(u,v)|u,v\in V(\overrightarrow{G})\}$. Additionally, we define $\theta(u,v)=\max\{\partial(u,v), \partial(v,u)\}$. 

Other notations that are not defined here can be found in the reference \cite{JAB}.

In 1939, Robbins \cite{HR} gave the sufficient and necessary condition: A connected undirected graph $ G $ has a strong orientation if and only if $ G $ is bridgeless. 
Then the \emph{oriented diameter} of a bridgeless graph $G$ is defined as: $$\overrightarrow{diam}(G)=\min\{d(\overrightarrow{G})\mid  \overrightarrow{G} \text{ is a strong orientation of }G\}.$$ 

The study of oriented diameter has been a central topic in graph theory since Robbins' foundational result. Over the past decades, many researchers studied oriented diameter and tried to obtain the sharp upper bounds for some other parameters or special classes of graphs. For instance, Fomin et al. \cite{FF} established an upper bound on the oriented diameter of a graph $ G $, proving that $\overrightarrow{diam}(G)\leq 9\gamma(G) - 5$, where $ \gamma(G) $ denotes the domination number of~$ G $, and Kurz and Lätsch \cite{SK} improved their upper bound to $\overrightarrow{diam}(G)\leq 4\gamma(G)$. Gutin \cite{GG} and Plesník \cite{JP} revealed that the oriented diameter of any complete $k$-partite ($k\geq 3$) graph lies between 2 and 4. Šoltés \cite{LS} showed that the oriented diameter of any complete bipartite graphs is at most $4$. Dankelmann et al. \cite{PD} demonstrated that $ \overrightarrow{diam}(G) \leq n - \Delta + 3 $, and this bound is sharp, where $\Delta$ is the maximum degree of $G$. Based on the girth $g(G)$ of $G$, Chen and Chang \cite{BC} showed that $ \overrightarrow{diam}(G) \leq n - \left\lfloor \frac{g(G) - 1}{2} \right\rfloor (\Delta - 4) - 1$, and this bound is tight for $\Delta\geq 4$. Bau and Dankelmann \cite{SB} showed that $ \overrightarrow{diam}(G) \leq \frac{11n}{\delta + 1} +9$, and Surmacs \cite{MS} improved their result and obtained $ \overrightarrow{diam}(G) < \frac{7n}{\delta + 1} $, where $\delta$ is the minimum degree of $G$. Wang et al. \cite{XW2} showed that for any maximal outplanar graph of order $ n \geq 3 $ with four exceptions, $ \overrightarrow{diam}(G) \leq \left\lceil \frac{n}{2} \right\rceil $, and the bound is tight. For a survey of other prior research in this area, refer to \cite{KK}.

	Let $f(d)$ be the smallest value for which every bridgeless graph $G$ with $d(G)=d$ admits a strong orientation $\overrightarrow{G}$ satisfying $d(\overrightarrow{G}) \leq f(d)$. In 1978, Chvátal and Thomassen \cite{VC} introduced $f(d)$, confirmed its existence, and obtained the following result. 

\begin{theorem}{\rm(\!\!\cite{VC})}\label{t1}
	Let $G$ be a bridgeless graph with $d(G)=d$. Then $$\frac12 d^2+d\leq f(d)\leq 2d^2+2d.$$
\end{theorem}

If $d\in\{1,2,3,4\}$, then $2\leq f(1)\leq 4$, $4\leq f(2)\leq 12$, $8\leq f(3)\leq 24$ and $12\leq f(4)\leq 40$ by Theorem \ref{t1}, respectively. In 1980, Boesch and Tindell \cite{FB} showed $f(1)=3$. In 1978, Chvátal and Thomassen \cite{VC} proved $f(2)=6$. In 2010, Kwok et al. \cite{PK} obtained $9\le f(3)\le 11$; subsequently, Wang and Chen \cite{XW} determined $f(3)=9$ in 2022. Recently, Babu et al. \cite{BD} obtained $f(4)\leq 21$, and $f(d)\leq1.373d^2 +6.971d-1$, which is smaller than $2d^2+2d$ when $d\geq 8$. 

For a bridgeless graph $G$ and any $e\in E(G)$, there exists at least a cycle containing $e$ in $G$. Then we define the edge girth of a bridgeless graph $G$ as follows.

\begin{definition}\label{def1}
	Let $G$ be a bridgeless graph and $e\in E(G)$. The \emph{edge girth} of $e$ is defined as the length of the shortest cycle containing $e$ in $G$, denoted by $l_{G}(e)$, and the \emph{edge girth} of $G$ is defined as the maximum of the edge girth over all edges in $G$, denoted by $g^*(G)$, that is, $ g^*(G) = \max \left\{ l_{G}(e) \mid e \in E(G) \right\}.$
\end{definition}

Clearly, $g(G)\le g^*(G)$ since the girth of $G$, $g(G)=\min\{\,l_G(e)\mid e\in E(G)\,\}$. For a bridgeless graph~$G$ with $d(G)=d$, every edge is contained in a cycle of length at most $2d+1$; otherwise, $d(G)\ge d+1$, a contradiction. 
Since $G$ may have parallel edges, we obtain $2\le g^*(G)\le 2d+1$.

In this paper, we focus on the oriented diameter of a bridgeless graph $G$ with diameter $d(G)=4$ in terms of the edge girth of $G$, and obtain the following results.

\begin{theorem}\label{t2}
	Let $G$ be a bridgeless graph with $d(G)=d$ and $g^*(G)\in \{2,3\}$. Then \\
	{\rm (i)} $\overrightarrow{diam}({G})=d$ if $g^*(G)=2$.\\
	{\rm (ii)} $\overrightarrow{diam}({G})\leq 3d$ if $g^*(G)=3$.
\end{theorem}
\begin{theorem}\label{t3}
	Let $G$ be a bridgeless graph with $d(G)=4$ and $g^*(G)\in \{6,7,8,9\}$. Then\\ {\rm (i)} $\overrightarrow{diam}(G)\leq 12$ if $g^*(G)=9$.\\
	{\rm (ii)} $\overrightarrow{diam}(G)\leq 13$ if $g^*(G)\in \{6,7,8\}$.
\end{theorem}

Let $d$, $g^*$ be positive integers with $2\leq g^*\leq 2d+1$. Figures \ref{fig5}-\ref{fig6} show that there exists a bridgeless graph $H_{d,g^*}$ with $d(H_{d,g^*})=d$ and $g^*(H_{d,g^*})=g^*$. 
Thus we can define:
$$F(d,g^*)=\max\{\overrightarrow{diam}(G)\mid G \text{ is a bridgeless graph with } d(G)=d \text{ and } g^*(G)=g^*\}.$$


\begin{figure}[htbp]
	\centering
	
	\begin{minipage}[b]{0.6\linewidth}
		\centering
		\includegraphics[width=\linewidth]{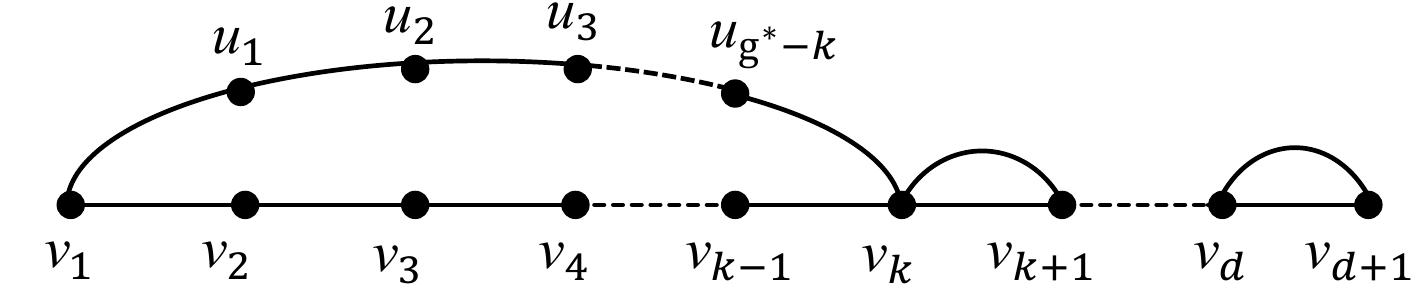}
		\caption{$H_{d,g^*}$ with $g^*\geq 3$ and $k=\lfloor\frac{g^*}{2}\rfloor+1$.}
		\label{fig5}
	\end{minipage}
	\hfill   
	\begin{minipage}[b]{0.35\linewidth}
		\centering
		\includegraphics[width=\linewidth]{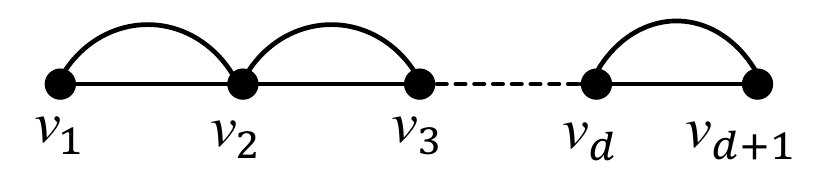}
		\caption{$H_{d,g^*}$ with $g^*=2$.}
		\label{fig6}
	\end{minipage}
\end{figure}

It is easy to check $f(d)=\max\{\overrightarrow{diam}(G)\mid G \text{ is a bridgeless graph with }d(G)=d\}$. For a bridgeless graph $G$ with $d(G)=d$, we have $2\leq g^*(G)\leq 2d+1$, and thus $f(d)=\max\{F(d,g^*)\mid 2\leq g^*\leq 2d+1\}$. Therefore, the research on $f(d)$ can reduce to studying $F(d,g^*)$ for $g^*=2,3,\dots,2d+1$. When $d=4$, we obtain the following result on $F(4,g^*)$ based on Theorems \ref{t2} and \ref{t3}.


\begin{theorem}\label{t5}
	$F(4,2)=4$, $F(4,9)=12$, $F(4,3)\leq 12$ and $F(4,g^*)\leq 13$ for $g^*\in \{6,7,8\}$. 
\end{theorem}

The present paper is organized in the following way. In Section 2, we provide some notations and lemmas. In Section 3, we give the proofs of Theorems \ref{t2}-\ref{t5} in order.

\section{Notations and Lemmas}\label{sec-pre}

\hspace{1.5em}Before proving Theorems \ref{t2}-\ref{t5}, we introduce the following notations and lemmas, which are useful in the proofs.

Let $G$ be a bridgeless graph. For $v\in V(G)$ and $U\subseteq V(G)$, we take $N(v)=\{u\ |\ uv\in E(G)\}$ and $N(U)= \bigcup\limits_{v \in U} N(v) - U$. Let $U$ and $V$ be two subsets of $V(G)$. We use $[U,V]$ to denote the set of all edges between $U$ and $V$. If $U=\{u\}$ or $V=\{v\}$, we write $[u,V]$ or $[U,v]$ for short. The notation $U\rightarrow V$ represents all edges in $[U,V]$ are oriented from $U$ to $V$. We write $u\rightarrow V$ or $U\rightarrow v$ if $U=\{u\}$ or $V=\{v\}$. Let $ u \in V(G) $ and $ S \subseteq V(G) $ with $ u \notin S $. A $ (u, S) $-path is defined as the shortest path from vertex $ u $ to any vertex in $ S $, and $ d(u, S) $ denotes the length of such a path. A directed $ (u, S) $-path is the shortest directed path from $ u $ to $ S $, and $ \partial(u, S) $ denotes the length of this directed path. Similarly, a directed $(S,u)$-path and its length $\partial(S,u)$ are also defined. Since $ \partial(u, S) $ and $ \partial(S, u) $ can differ, we define $ \theta(u, S) $ as the maximum of these two values, that is, $ \theta(u, S) = \max\{\partial(u, S), \partial(S, u)\} $.

Kwok et al.\ used a special orientation in the proof of \cite[Lemma 3.4]{PK} in 2010. Subsequently, in 2022, Wang and Chen~\cite{XW} improved it and named it the \emph{$R$-$S$ orientation}, whose definition is given below. 
\begin{definition}{\rm(\!\!\cite{XW,PK})}\label{def2}
	Let $G$ be a graph, and let $ R $ and $ S $ be two disjoint vertex subsets of $G$ such that $ G[S] $ has no isolated vertex and $ N(w) \cap R \neq \emptyset $ for every $ w \in S $. Suppose $ F $ is a spanning forest of $ G[S] $ with bipartition $ V_1 $ and $ V_2 $, where $F$ contains no isolated vertex. Orient the edges of $ [R, V_1] $, $ [V_1, V_2] $ and $ [V_2, R] $ as $ R \rightarrow V_1 \rightarrow V_2 \rightarrow R $. This orientation is called $R$-$S$ orientation $($see Figure \ref{fig4}$)$.
\end{definition}

\begin{figure}[h]
	\centering
	\includegraphics[scale=0.24]{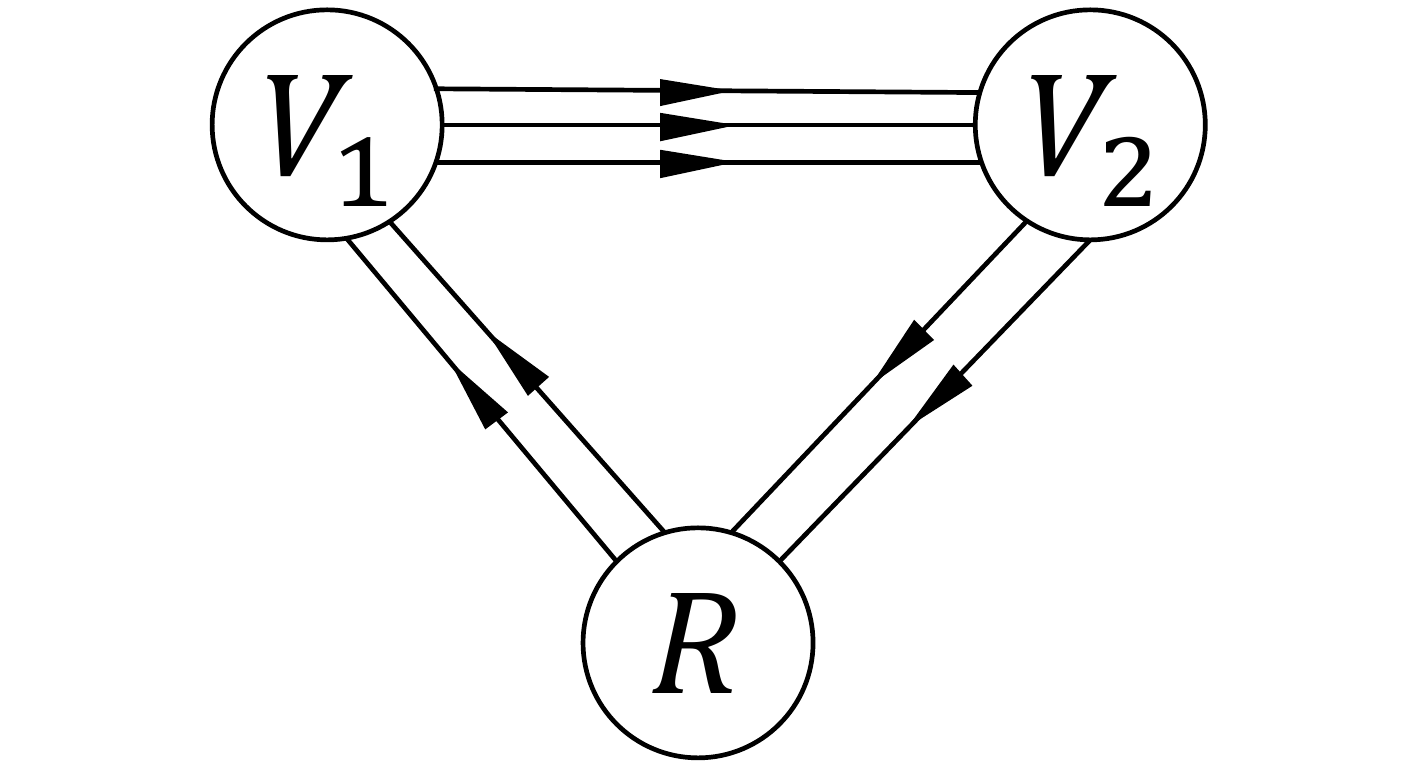}
	\caption{$R$-$S$ orientation.}\label{fig4}
\end{figure}
\begin{lemma}{\rm(\!\!\cite{PK})}\label{l1}
$ \theta(w, R) \leq 2 $ for every $ w \in S $ in an $ R $-$ S $ orientation. 
\end{lemma}  

By Definition \ref{def2}, we have the following observation.

\begin{observation}\label{l1.5}
	For an $ R $-$ S $ orientation, if $ F $ is a spanning forest of $ G[S] $ with bipartition $ V_1 $ and $ V_2 $, and $R\rightarrow V_1\rightarrow V_2\rightarrow R$, then $ \partial(R,w) =1,\partial(w,R)=2 $ for any $w\in V_1$, and $\partial(R,w)=2, \partial(w,R) =1 $ for any $w\in V_2$.
\end{observation}

\begin{lemma}{\rm(\!\!\cite{VC})}\label{l2}
	Every simple graph $G$ admits an orientation $\overrightarrow{G}$ with the following property: If an edge $uv$ belongs to a cycle of length $k$ in $G$, then $\overrightarrow{uv}$ or $\overrightarrow{vu}$ belongs to a directed cycle of length at most $h(k)$ in $\overrightarrow{G}$, where $h(k) = (k - 2) \cdot 2^{\lfloor (k - 1)/2 \rfloor} + 2$.
\end{lemma}

\section{Proofs of Main Results}
\hspace{1.5em} In this section, we give the proofs of Theorems \ref{t2}-\ref{t5}.

\vspace{0.5\baselineskip}
{\noindent\textbf{\bf {Proof of Theorem \ref{t2}.}}} Let $G$ be a bridgeless graph with $d(G)=d$. 

(i) If $g^*(G)=2$, then each edge of $G$ has parallel edges. For any $uv\in E(G)$, we orient the edges between $u$ and $v$ in two ways. Thus we obtain an orientation $\overrightarrow{G}$ of $G$.

For $u,v \in V(G)$, if $uv\in E(G)$, then $\theta(u,v)=1$. Otherwise, there exists a shortest $(u,v)$-path $P=uw_1w_2\dots w_t v$ with $1\leq t\leq d-1$ in $G$ by $d(G)=d$. Then $\theta(u,v) \leq \theta(u,w_1)+\theta(w_1,w_2)+\dots+\theta(w_t,v)=t+1\leq d$. Then $d(\overrightarrow{G})\leq d$ by the definition of $d(\overrightarrow{G})$, and thus $\overrightarrow{diam}(G)=d$ by $d(\overrightarrow{G})\geq \overrightarrow{diam}(G)\geq d(G)=d$.

(ii) If $g^*(G)=3$, then each edge of $G$ either has parallel edges or is contained in a triangle of $G$. Let $H$ be a simple subgraph obtained from $G$ by deleting all parallel edges of $G$. Then $H$ is connected, and $H$ contains at least one triangle by $g^*(G)=3$. Thus there is an orientation $\overrightarrow{H}$ such that $\overrightarrow{uv}$ or $\overrightarrow{vu}$ belongs to a directed cycle of length at most $4$ in $\overrightarrow{H}$ by Lemma \ref{l2}, where $uv\in E(H)$ belongs to a triangle of $H$.

Now we construct an orientation $\overrightarrow{G}$ of $G$ based on $\overrightarrow{H}$. We only need to consider the edges belonging to $E(G)-E(H)$. For each $e=xy\in E(G)-E(H)$, $e$ must be a parallel edge of $G$. Then there exists an edge $e'=xy\in E(H)$, and we orient $e$ as the reverse of the orientation of $e'$. 

For any two vertices $u,v\in V(G)$, if $uv\in E(G)$, then $\overrightarrow{uv}$ or $\overrightarrow{vu}$ belongs to a directed cycle of length at most $4$ in $\overrightarrow{G}$, and thus $\theta(u,v)\leq 3$; if $uv\notin E(G)$, then there is a $(u,v)$-path $P=uw_1\dots w_t v$ with $1\leq t\leq d-1$ in $G$ by $d(G)=d$. Then $\theta(u,v)\leq \theta(u,w_1)+\dots+\theta(w_t,v)\leq 3(t+1)\leq 3d$. Therefore, $\overrightarrow{G}$ is an strong orientation, and $\overrightarrow{diam}(G)\leq d(\overrightarrow{G})\leq 3d$. {\hfill $\blacksquare$\par}\vspace{0.3\baselineskip}

Before proving (i) and (ii) of Theorem \ref{t3}, we introduce some notations.

Let $G$ be a bridgeless graph with $d(G)=4$ and $g^*(G)\in\{6,7,8,9\}$. By Definition \ref{def1}, there is an edge $e=uv\in E(G)$ such that $l_G(e)=g^*(G)$. 

Let $ S_{i,j} = \{ w \in V(G) \mid d(w, u) = i \text{ and } d(w, v) = j \} $, where $1\leq i,j\leq d(G)=4$. 
 By $l_G(uv)\in \{6,7,8,9\}$, we have $S_{1,1}=S_{2,2}=\emptyset$, and thus the sets $ \{u, v\}, S_{1,2}, S_{2,1}, S_{2,3}, S_{3,2},$ $ S_{3,3}, S_{3,4}, S_{4,3},  S_{4,4}$ are a partition of $ V(G) $. Now we define a partition of $ S_{i,j} $ except $ S_{3,3} $ (as shown in Figure \ref{fig8}) as follows:

$$
\begin{aligned}
	A' &= \{ w \mid w \in S_{1,2} \text{ and } N(w) \subseteq S_{1,2} \cup \{u\} \}, & A &= S_{1,2} - A'; \\
	B' &= \{ w \mid w \in S_{2,1} \text{ and } N(w) \subseteq S_{2,1} \cup \{v\} \}, & B &= S_{2,1} - B'; \\
	I' &= \{ w \mid w \in S_{2,3} \text{ and } N(w) \subseteq S_{2,3} \cup A \}, & I &= S_{2,3} - I'; \\
	J' &= \{ w \mid w \in S_{3,2} \text{ and } N(w) \subseteq S_{3,2} \cup B \}, & J &= S_{3,2} - J'; \\
	K' &= \{ w \mid w \in S_{3,4} \text{ and } N(w) \subseteq S_{3,4} \cup I \}, & K &= S_{3,4} - K'; \\
	L' &= \{ w \mid w \in S_{4,3} \text{ and } N(w) \subseteq S_{4,3} \cup J \}, & L &= S_{4,3} - L'; \\
	M' &= \{ w \mid w \in S_{4,4} \text{ and } |[w, S_{3,3} \cup K \cup L]| = 1 \}, & M &= S_{4,4} - M'.
\end{aligned}
$$

\setlength{\intextsep}{3pt}   
\setlength{\abovecaptionskip}{2pt} 
\begin{figure}[h]
	\centering
	\includegraphics[scale=0.6]{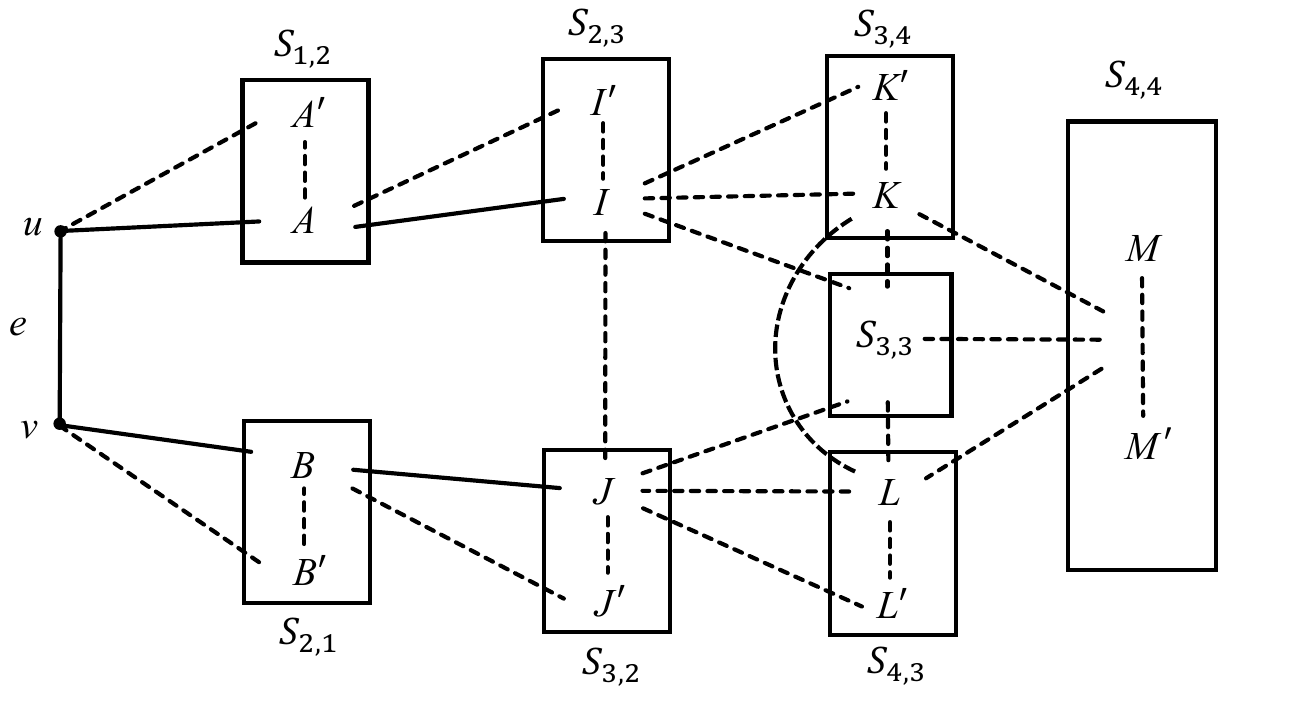}
	\caption{The bridgeless graph $G$ with $l_G(e)=g^*(G)\in\{6,7,8,9\}$.}\label{fig8}
\end{figure}
\begin{proposition}\label{prop1}
	For $w\in S_{4,4}$, we have \\
	{\rm (i)} If $N(w)\cap S_{3,3}=\emptyset$, then $N(w)\cap K\neq \emptyset$ and $N(w)\cap L\neq \emptyset$.\\
	{\rm (ii)} If $w\in M'$, then $|N(w)\cap S_{3,3}|=1$ and $N(w)\cap (K\cup L)=\emptyset$.
\end{proposition}

\begin{proof}
	If $N(w)\cap S_{3,3}= \emptyset$, then $N(w)\cap K\neq \emptyset$ and $N(w)\cap L\neq \emptyset$ by $d(w,u)=d(v,w)=4$.
	
	If $w\in M'$ and $N(w)\cap S_{3,3}$
	=$\emptyset$, then $|[w,S_{3,3}\cup K\cup L]|\geq 2$ by (i), a contradiction with the definition of $M'$. Thus $N(w)\cap S_{3,3}\neq\emptyset$, and (ii) holds immediately.
\end{proof}

\vspace{0.3\baselineskip}
{\noindent\textbf{\bf {Proof of (i) of Theorem \ref{t3}.}}} By the above definitions and $l_G(uv)=g^*(G)=9$, we have $A,B,I,J,L,K,S_{4,4}\neq \emptyset$ and $[A,B]=[I,J]=[K,L]=S_{3,3}=\emptyset$. If $M'\neq\emptyset$, then there exists a vertex $w\in M'$, and thus $N(w)\cap S_{3,3}\neq\emptyset$ by (ii) of Proposition \ref{prop1}, a contradiction with $S_{3,3}=\emptyset$. Thus $M'=\emptyset$ and $M=S_{4,4}$. 

Since $d(G)=4$ and $L\neq \emptyset$, we have $A'=\emptyset$. Otherwise, for any $x\in L$ and $y\in A'$, we have $d(x,y)\geq 5$, a contradiction with $d(G)=4$. Similarly, $B'=I'=J'=K'=L'=\emptyset$ since $d(G)=4$ and $K,J,I,B,A$ are not empty, respectively.

By the above arguments, we have that the sets $\{u,v\},A,I,K,M,L,J,B$ are a partition of $V(G)$, and some adjacency relationships of vertices in different partition blocks can be obtained from the definition of $S_{i,j}$ for $1\leq i,j\leq 4$. Now we construct an orientation $\overrightarrow{G}$ of $G$ as follows.

\vspace{0.5\baselineskip}
{\noindent\textbf{\bf {Construction I.}}}
{\itshape Let $\overrightarrow{G}$ be an orientation of $G$: $M\rightarrow K\rightarrow I\rightarrow A\rightarrow u\rightarrow v\rightarrow B\rightarrow J\rightarrow L\rightarrow M$, and for undirected edges in $G$, we orient these edges arbitrarily.}
\vspace{0.5\baselineskip}

By Construction I, we obtain some known directed distances as shown in Table \ref{Table1}.
\vspace{0.3\baselineskip}
\begin{table}[h]
	\centering
	\label{Table1}
	\renewcommand{\arraystretch}{1.5}
	\setlength{\tabcolsep}{10pt}
	\begin{tabular}{|c|c|c|c|c|c|c|c|}
		\hline
		\textbf{for \(w\) in} & \(A\) & \(B\) & \(I\) & \(J\) & \(K\) & \(L\) & \(M\)  \\ \hline
		\(\partial(w, u)\) & 1 & 7 & 2 & 6 & 3 & 5 & 4  \\ \hline
		\(\partial(v, w)\) & 7 & 1 & 6 & 2 & 5 & 3 & 4  \\ \hline
	\end{tabular}
	\vspace{3pt}
		\caption{Some known directed distances with $l_G(uv)=g^*(G)=9$.}
\end{table}

Clearly, Claim I holds by $\partial(x,y)\leq \partial(x,u)+\partial(u,v)
+\partial(v,y)$ and $\partial(u,v)=1$.

\vspace{0.5\baselineskip}
{\noindent\textbf{\bf {Claim I.}}}
	{\itshape If $x\in \{u,v\}\cup A\cup I\cup K\cup M$ or $y\in \{u,v\}\cup B\cup J\cup L\cup M$, then $\partial(x,y)\leq 12$.}
\vspace{0.5\baselineskip}

Now, we only need to consider $x\in B\cup J\cup L$ and $y\in A\cup I\cup K$ by the following four cases. 

{\noindent\textbf{\bf {Case 1.}}} $x\in J,y\in I$, or $x\in L,y\in A$, or $x\in B,y\in K$.

In this case, we have $d(x,y)=4$. Let $P=xw_1w_2w_3y$ be the shortest $(x,y)$-path in $G$. By Construction I, we have $x\rightarrow w_1\rightarrow w_2\rightarrow w_3\rightarrow y$, and thus $\partial(x,y)=4$ in $\overrightarrow{G}$. 

{\noindent\textbf{\bf {Case 2.}}} $x\in B,y\in I$ or $x\in J,y\in A$.

If $x\in B$ and $y\in I$, then $N(x)\cap J\neq \emptyset$ by the definition of $B$. Let $w\in N(x)\cap J$. Then we have $\partial(x,y)\leq \partial(x,w)+\partial(w,y)\leq 1+4=5$ by Construction I and Case 1. Similarly, if $x\in J$ and $y\in A$, we have $\partial(x,y)\leq 5$. 

{\noindent\textbf{\bf {Case 3.}}} $x\in B,y\in A$.

In this case, we have $N(x)
\cap J\neq \emptyset$ and $N(y)\cap I\neq \emptyset$. Let $w_1\in N(x)\cap J$ and $w_2\in N(y)\cap I$. Then we have $\partial(x,y)\leq \partial(x,w_1)+\partial(w_1,w_2)+\partial(w_2,y)\leq 1+4+1=6$ by Construction I and Case 1. 

{\noindent\textbf{\bf {Case 4.}}} $x\in L,y\in I\cup K$ or $x\in J,y\in K$.

If $x\in L,y\in I$ or $x\in J,y\in K$, then $\partial(x,y)\leq \partial(x,u)+\partial(u,v)
+\partial(v,y)\leq 12$ by Table \ref{Table1} and $\partial(u,v)=1$. Similarly, if $x\in L,y\in K$, then $\partial(x,y)\leq \partial(x,u)+\partial(u,v)
+\partial(v,y)\leq 11$.

Combining Claim I and Cases 1-4, we obtain $\overrightarrow{diam}(G)\leq d(\overrightarrow{G})\leq 12$.{\hfill $\blacksquare$\par}

\vspace{0.5\baselineskip}
{\noindent\textbf{\bf {Proof of (ii) of Theorem \ref{t3}.}}} By the above definitions and $l_G(uv)=g^*(G)\in\{6,7,8\}$, we see $A,B,I,J\neq \emptyset$. Our aim is to obtain an orientation of $G$ whose diameter is at most $13$. To achieve the required structural properties, we further partition each of the seven sets $K,L,I,J,A,B,M'$ into smaller subsets as follows.
\begin{gather*}
	K_1 = K \cap N(S_{3,3} \cup L), \quad K_2 = K - K_1; \quad
	L_1 = L \cap N(S_{3,3} \cup K), \quad L_2 = L - L_1; \\
	I_1 = I \cap N(J), \quad I_2 = I \cap N(S_{3,3}) - I_1, \quad I_3 = I \cap N(K_1) - I_1 - I_2, \\
	I_4 = \{w \mid w \in I - \cup_{i=1}^{3} I_i, \, N(w) \cap I_1 \neq \emptyset\}, \quad I_5 = \{w \mid w \in I - \cup_{i=1}^{4} I_i, \, N(w) \cap I_2 \neq \emptyset\}, \\
	I_6 = \{w \mid w \in I - \cup_{i=1}^{5} I_i, \, N(w) \cap I_4 \neq \emptyset\}, \quad I_7 = I \cap N(K_2) - \cup_{i=1}^{6} I_i, \quad I_8 = I - \cup_{i=1}^{7} I_i; \\
	J_1 = J \cap N(I), \quad J_2 = J \cap N(S_{3,3}) - J_1, \quad J_3 = J \cap N(L_1) - J_1 - J_2, \\
	J_4 = \{w \mid w \in J - \cup_{i=1}^{3} J_i, \, N(w) \cap J_1 \neq \emptyset\}, \quad J_5 = \{w \mid w \in J - \cup_{i=1}^{4} J_i, \, N(w) \cap J_2 \neq \emptyset\}, \\
	J_6 = \{w \mid w \in J - \cup_{i=1}^{5} J_i, \, N(w) \cap J_4 \neq \emptyset\}, \quad J_7 = J \cap N(L_2) - \cup_{i=1}^{6} J_i, \quad J_8 = J - \cup_{i=1}^{7} J_i; \\
	A_1=A\cap N(I_1),\quad A_j = A \cap N(I_j) - \cup_{i=1}^{j-1} A_i \text{ for } j = 2, \ldots, 8, \\
	A_9 = \{w \mid w \in A - \cup_{i=1}^{8} A_i, \, N(w) \cap A_1 \neq \emptyset\}, \quad A_{10} = A - \cup_{i=1}^{9} A_i; \\
	B_1=B\cap N(J_1),\quad B_j = B \cap N(J_j) - \cup_{i=1}^{j-1} B_j \text{ for } j = 2, \ldots, 8, \\
	B_9 = \{w \mid w \in B - \cup_{i=1}^{8} B_i, \, N(w) \cap B_1 \neq \emptyset\},\quad B_{10} = B - \cup_{i=1}^{9} B_i; \\
	M'_1 = \{w \mid w \in M', \, N(w) \cap M \neq \emptyset\}, \\
	M'_2 = \{w \mid w \in M' - M'_1, \, w \text{ is isolated in } G[M' - M'_1]\}, \\
	M'_3 = \{w \mid w \in M' - M'_1, \, w \text{ is not isolated in } G[M' - M'_1]\}.
\end{gather*}

By the definitions of $A_{10},B_{10},I_8,J_8,K_2,L_2$, we have the following observation.

\begin{observation}\label{obs1}
	The neighborhoods of the vertices in $A_{10},B_{10},I_8,J_8,K_2,L_2$ are characterized in Table~\ref{Table 2}.
	\vspace{0.25cm}
	\begin{table}[h]
		\centering
		\small
		\setlength{\tabcolsep}{2pt}
		\renewcommand{\arraystretch}{1.3}
		\begin{tabular}{|c|c|c|c|c|c|c|}
			\hline
			for $w$ in & $A_{10}$ & $B_{10}$ & $I_8$ & $J_8$ & $K_2$&  $L_2$ \\ \hline
			$N(w)\subseteq$ & $\{u\}\cup S_{1,2}\cup I'$ & $\{v\}\cup S_{2,1}\cup J'$ & $A\cup S_{2,3}\cup K'$ & $B\cup S_{3,2}\cup L'$ & $I\cup S_{3,4}\cup M$ & $J\cup S_{4,3}\cup M$ \\ \hline
		\end{tabular}
		\vspace{0.125cm}
		\caption{Neighborhoods of some vertices.}
		\label{Table 2}
	\end{table}
	\vspace{-10pt}
\end{observation}

We now begin to construct the desired orientation of the graph $G$ through the following stepwise procedure.
\begin{construction}\label{con1}
	Let \( A' \rightarrow A \rightarrow u \rightarrow v \rightarrow B \rightarrow B' \), \( J \rightarrow I \cup S_{3,3}\cup L \),  $L\rightarrow M\cup S_{3,3}$, $L\cup S_{3,3}\cup M\rightarrow K$, $S_{3,3}\cup K\rightarrow I$, as shown in Figure \ref{fig1}; $(I-I_8)\rightarrow A$, $B\rightarrow (J-J_8)$; \( A_i \rightarrow A_j \) for \( 1 \leq i < j \leq 10 \) and \( B_j \rightarrow B_i \) for \( 1 \leq i < j \leq 10 \); \( I_i \rightarrow I_j \) for \( 1 \leq i < j \leq 8 \) and \(J_j \rightarrow J_i \) for \( 1 \leq i < j \leq 8 \); \( K_1 \rightarrow K_2 \) and \( L_2 \rightarrow L_1 \).
\end{construction}
\begin{figure}[h]
	\centering
	\includegraphics[scale=0.6]{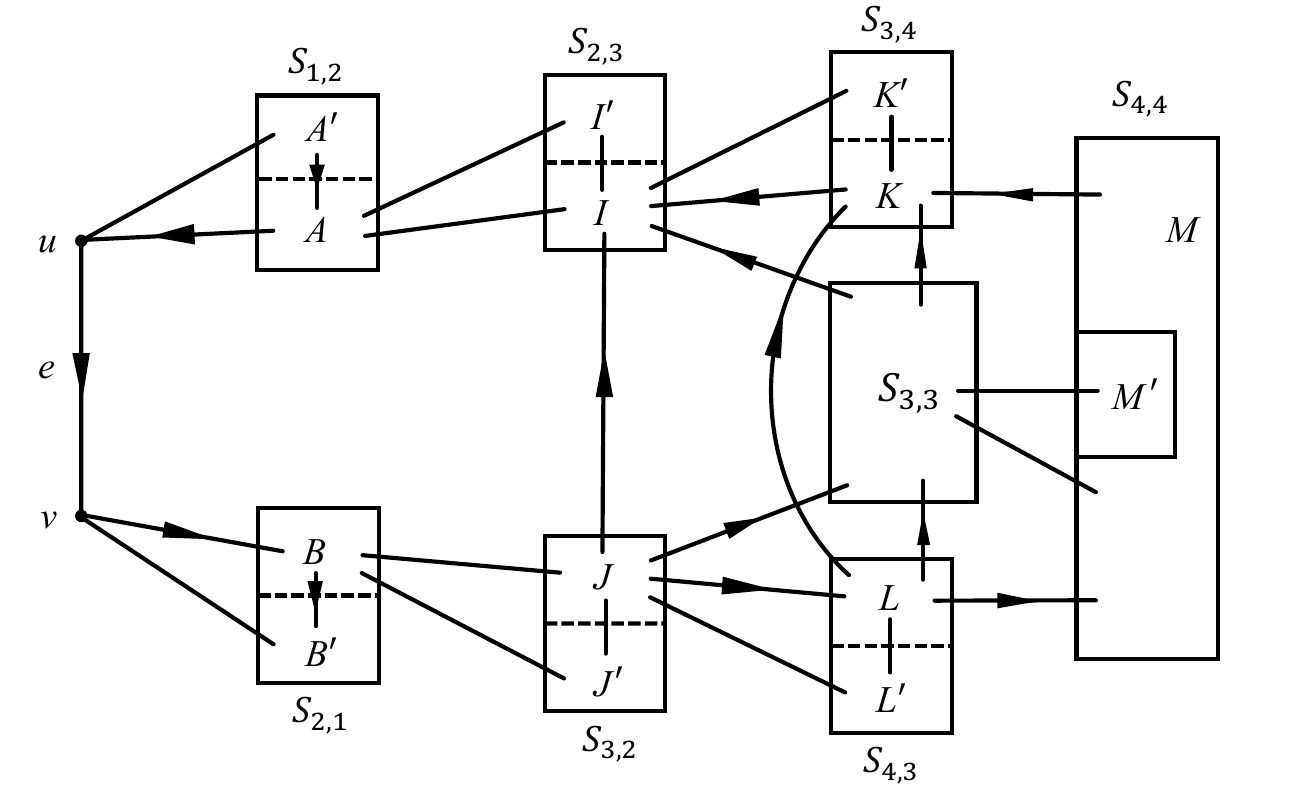}
	\caption{Orientations of some edges of $G$ with $g^*(G)\in \{6,7,8\}$.}\label{fig1}
\end{figure}

For $w\in S_{3,3}$, by the definitions of $I_1,I_2,J_1,J_2,S_{3,3}$, we have $N(w)\cap (I_1\cup I_2)\neq\emptyset$, $N(w)\cap (J_1\cup J_2)\neq\emptyset$. By Construction \ref{con1}, we have $v\rightarrow B\rightarrow J_1\cup J_2\rightarrow S_{3,3}\rightarrow I_1\cup I_2\rightarrow A\rightarrow u$, and thus the following claim holds. 
\begin{claim}\label{claim9}
	$\partial(u,v)=1$, $\partial(v,w)=\partial(w,u)=3$ for any $w\in S_{3,3}$, and $\partial(a,u)=\partial(v,b)=1$ for any $a\in A$ and $b\in B$.
\end{claim}
\begin{construction}\label{con2}	
	Let $S$ be the set of all non-isolated vertices of $G[A']$ or $G[B']$, and $R$ be $\{u\}$ or $\{v\}$, respectively. Then we orient some edges of $G[\{u\}\cup A']$ or $G[\{v\}\cup B']$ as follows:\\  {\rm (i)} For the edges of $G[\{u\}\cup S]$ or $G[\{v\}\cup S]$, we give an $R$-$S$ orientation to the edges of $G[S]$ and $[R,S]$. \\	
	{\rm (ii)} For $w\in A'- S$ or $w\in B'- S$, $w$ is an isolated vertex of $G[A']$ or $G[B']$, respectively.  
	
	\quad $\circ\ $ If $wu$ or $wv$ has parallel edges, we orient the edges of $wu$ or $wv$ in two ways.  
	
	\quad $\circ\ $ If $wu$ or $wv$ has no parallel edges, then $w$ must have a neighbor in $A$ or $B$ since $G$ is bridgeless, and we orient the edge as $\overrightarrow{uw}$ or $\overrightarrow{wv}$, respectively.
\end{construction}

By Construction \ref{con2}, Definition \ref{def2} and Lemma \ref{l1}. we have the following.

\begin{claim}\label{claim10}
	For any $a\in A'$ and $b\in B'$, we have $\max\{\theta(a,u),\theta(b,v)\}\leq 2$. 
\end{claim}

\begin{construction}\label{con11}
	For any $w \in M$, we orient the edges of $[w, S_{3,3}]$ in two ways if $|[w, S_{3,3}]| \geq 2$, as $S_{3,3} \rightarrow w$ if $|[w, S_{3,3}]| = 1$ and $N(w) \cap K \neq \emptyset$, as $w \rightarrow S_{3,3}$ if $|[w, S_{3,3}]| = 1$, $N(w) \cap K = \emptyset$ and $N(w) \cap L \neq \emptyset$.
\end{construction}

\begin{claim}\label{claim16}
	{\rm (i)} For any $a\in A'$ and $b\in B'$, we have $\max\{\partial(v,a),\partial(b,u)\}\leq 9$.
	\\{\rm (ii)} For any $w\in M$, we have $\partial(w,u)=4$ and $\partial(v,w)= 4$.
	\\{\rm (iii)} For any $w\in K_i$ with $i=1,2$, we have $\partial(w,u)= 3$ and $\partial(v,w)\leq 3+i$.\\{\rm (iv)} For any $w\in L_i$ with $i=1,2$, we have $\partial(v,w)= 3$ and $\partial(w,u)\leq 3+i$. \\
	{\rm (v)} $\partial(v,w)\leq 
	\begin{cases}
		 i+2, &\text{if }  w\in I_i,\  1\leq i\leq 3,\\
		 4, & \text{if } w\in I_4,\\
		 5, & \text{if } w\in I_5\cup I_6,\\
		 6, & \text{if } w\in I_7;
	\end{cases}$\  $\partial(w,u)\leq 
	\begin{cases}
		 i+2, &\text{if } w\in J_i,\ 1\leq i\leq 3,\\
		 4, &\text{if } w\in J_4,\\
		 5, &\text{if } w\in J_5\cup J_6,\\
		 6, & \text{if }w\in J_7.
	\end{cases}$
\end{claim}	

\begin{proof}
	By $l_G(uv)\leq 8$, Observation \ref{obs1} and Construction \ref{con1}, there exists a directed cycle of length $l_G(uv)$ containing $uv$. Then $\partial(v,u)= l_G(uv)-1\leq 7$, and thus (i) holds by Claim \ref{claim10}. 
	
	Let $w\in M$. Then $|[w,S_{3,3}\cup K\cup L]|\geq 2$, which implies $|[w,S_{3,3}]|\geq 2$, or $|[w,S_{3,3}]|=1$ with $N(w)\cap (K\cup L)\neq\emptyset$, or $|[w,S_{3,3}]|=0$ with $N(w)\cap K\neq\emptyset$ and $N(w)\cap L\neq\emptyset$ by (i) of Proposition \ref{prop1}. 
	
	If $|[w,S_{3,3}]|\geq 2$, then there exist vertices $w',w''\in S_{3,3}$ (not necessarily distinct) such that $w'\to w\to w''$ by Construction \ref{con11}, and hence $\partial(w,u)\leq \partial(w,w'')+\partial(w'',u)=1+3=4$ and $\partial(v,w)\leq \partial(v,w')+\partial(w',w)=3+1=4$ by Claim \ref{claim9},
	with equality since $d(w,v)=d(w,u)=4$. 
	
	If $|[w,S_{3,3}]|=1$ and $N(w)\cap K\neq\emptyset$, then $S_{3,3}\rightarrow w\rightarrow K\rightarrow (I-I_8)\rightarrow A\rightarrow u$ by Construction \ref{con1}; if $|[w,S_{3,3}]|=1$, $N(w)\cap K=\emptyset$ and $N(w)\cap L\neq\emptyset$, then $v\rightarrow B\rightarrow (J-J_8)\rightarrow L\rightarrow w\rightarrow S_{3,3}$ by Construction \ref{con1}; if $|[w,S_{3,3}]|=0$ with $N(w)\cap K\neq\emptyset$ and $N(w)\cap L\neq\emptyset$, then $v\rightarrow B\rightarrow (J-J_8)\rightarrow L\rightarrow w\rightarrow K\rightarrow (I-I_8)\rightarrow A\rightarrow u$ by Construction \ref{con1}. Combining the above arguments, we have (ii) holds by Claim \ref{claim9}.
	
	Let $w\in K$. If $w\in K_1$, then $N(w)\cap (S_{3,3}\cup L)\neq\emptyset$, and thus we have $v\rightarrow B\rightarrow (J-J_8)\rightarrow (S_{3,3}\cup L)\rightarrow w\rightarrow (I-I_8)\rightarrow A\rightarrow u$ by Construction \ref{con1}. If $w\in K_2$, then $N(w)\cap M\neq\emptyset$, and thus we have $M\rightarrow w\rightarrow (I-I_8)\rightarrow A\rightarrow u$ by Construction \ref{con1}. Combining the above arguments and (ii), we have (iii) holds. Similarly, (iv) holds.
	
	Let $w\in I$. If $w\in I_1$, then $v\rightarrow B\rightarrow J_1\rightarrow w$ by Construction \ref{con1}. If $w\in I_4$, then $N(w)\cap I_1\neq\emptyset$, and thus $v\rightarrow B\rightarrow J_1\rightarrow I_1\rightarrow w$ by Construction \ref{con1}. Combining the above arguments, the result holds for $w\in I_1\cup I_4$. Similarly, (v) holds.
\end{proof}

	By Proposition \ref{prop1} and the definitions of $M'$ and $M'_2$, we have $|N(w)\cap S_{3,3}|=1$ for any $w\in M'$, and $N(w)\cap M'_1\neq\emptyset$ for any $w\in M'_2$ since $G$ has no bridges. By the definition of $M'_3$, we see that $G[M'_3]$ has no isolated vertices, and $N(w)\cap S_{3,3}\neq\emptyset$ for any $w\in M'_3$. Now we orient some edges from $M'$ as Construction \ref{con12}.

\begin{construction}\label{con12}
	 Let $S_{3,3}\rightarrow M'_2\rightarrow M'_1\rightarrow S_{3,3}$, $M\rightarrow M'_1$, $R=S_{3,3}$ and $S=M'_3$. Then we give an $R$-$S$ orientation for the edges of $G[S]$ and $[R,S]$. 
\end{construction}

\begin{claim}\label{claim17}
	For any $w\in M'$, we have $\partial(w,u)\leq 5$ and $\partial(v,w)\leq 5$. 
\end{claim}

\begin{proof}
	For any $w\in M'$, by Proposition \ref{prop1}, we have $N(w)\cap S_{3,3}\neq\emptyset$. 
	
	If $w\in M'_1$, then we have $N(w)\cap M\neq\emptyset$, and thus we have $w'\rightarrow w\rightarrow w''$ for $w'\in N(w)\cap M$ and $w''\in N(w)\cap S_{3,3}$ by Construction \ref{con12}. Therefore, $\partial(w,u)\leq \partial(w,w'')+\partial(w'',u)=1+3=4$ by Claim \ref{claim9}, and $\partial(v,w)\leq \partial(v,w')+\partial(w',w)=4+1=5$ by (ii) of Claim \ref{claim16}.
	
	If $w\in M'_2$, we have shown earlier that $N(w)\cap M'_1\neq\emptyset$, and thus  $S_{3,3}\rightarrow w\rightarrow w_1\rightarrow S_{3,3}$ for $w_1\in N(w)\cap M'_1$ by Construction \ref{con12}, which implies $\partial(w,u)\leq 5$ and $\partial(v,w)\leq 4$ by Claim \ref{claim9}. 
	
	If $w\in M'_3$, then by Construction \ref{con12}, Lemma \ref{l1} and Claim \ref{claim9}, we obtain the result. 
	
	Combining the above arguments, the result holds by $M'=M'_1\cup M'_2\cup M'_3$.
\end{proof}

To orient the edges of \([A, I']\), \([I, I']\) and $G[I']$, we introduce the following notations.

For \(w \in I'\), let \(P_w = w i_1 \cdots i_\ell\) denote a shortest \((w, J)\)-path in $G$, where $\ell\in \{2,3,4\}$ and \(i_\ell \in J\). If \(\ell = 2\), then \(i_1 \in I_1\). If \(\ell = 3\), then \(i_1 \in I' \cup A_1\cup (\cup_{i=2}^4 I_i)\) and \(i_2 \in I_1 \cup S_{3,3}\). If $\ell=4$, then $i_1\in I'\cup (\cup_{i=2}^{10}A_i)\cup I_3\cup (\cup_{i=5}^8 I_i)$, $i_2\in A_1\cup I'\cup (\cup_{i=2}^4 I_i)\cup K\cup K'$ and $i_3\in I_1\cup S_{3,3}\cup L_1$. Now we give a partition of \(I'\) such that $I'=\cup_{i=1}^8 I'_8$, and $I'_i\cap I'_j=\emptyset$ for $1\leq i<j\leq 8$, where
\begin{align*}
	I'_1 &= \{ w \mid w \in I' \text{ and } \forall\  P_w = w i_1 i_2 i_3 i_4\text{ with } i_1 \in\cup_{i=2}^{10}A_i, i_2 \in  I', i_3\in I_1\text{ and } i_4\in J  \}, \\
	I'_2 &= \{ w \mid w \in I' \text{ and } \exists\  P_w = w i_1 i_2 i_3i_4 \text{ with } i_1 \in I_3\cup(\cup_{i=5}^8 I_i), i_2 \in K_1\cup (\cup_{i=2}^4 I_i)\cup I',\\&\quad\ \ i_3\in L_1\cup S_{3,3}\cup I_1,i_4\in J,\text{ or } i_1\in I', i_2\in \cup_{i=1}^4 I_i, i_3\in I_1\cup S_{3,3},i_4\in J \}, \\
	I'_3 &= \{ w \mid w \in I'-I'_1 \text{ and } \forall\  P_w = w i_1 i_2 i_3i_4 \text{ with } i_1 \in I_7\cup I_8, i_2 \in A_1\cup K_2\cup K',\\&\quad\  i_3\in I_1 ,i_4 \in J, \text{ or }  i_1 \in\cup_{i=2}^{10}A_i, i_2 \in  I', i_3\in I_1, i_4\in J \}, \\
	I'_4 &= \{ w \mid w \in I' \text{ and } \exists\  P_w = w i_1 i_2 i_3 \text{ with } i_1 \in \cup_{i=2}^4 I_i, i_2 \in I_1\cup S_{3,3} \text{ and } i_3 \in J \}, \\
	I'_5 &= \{ w \mid w \in I' - \cup_{i=1}^4 I_i',\ w \text{ is isolated in } G[I' - \cup_{i=1}^4 I_i'] \}, \\
	I'_6 &= \{ w \mid w \in I' - \cup_{i=1}^5 I_i' \text{ and } d(w, J) = 2 \}, \\
	I'_7 &= \{ w \mid w \in I' - \cup_{i=1}^5 I_i' \text{ and } d(w, J) = 3 \}, \text{ and }\\
	I'_8 &= \{ w \mid w \in I' - \cup_{i=1}^5 I_i' \text{ and } d(w, J) = 4 \}.
\end{align*}

\begin{proposition}\label{prop1.1}
	{\rm (i)} If $J\neq J_1$, then $I'_1=I'_3=\emptyset$. {\rm (ii)} If $B\neq B_1$, then $I'_1=I'_3=\emptyset$.
\end{proposition}
\begin{proof}
	(i) Suppose $I'_1\cup I'_3\neq\emptyset$. Let $w\in I'_1\cup I'_3$. By the definitions of $I'_1$ and $I'_3$, we have $d(w,J)=4$, and thus $d(w,w')=4$ for any $w'\in J$ by $d(G)=4$. Let $P=ww_1w_2w_3w'$ be the shortest $(w,w')$-path. Then $P$ also is a shortest $(w,J)$-path. By the definitions of $I'_1$ and $I'_3$, we have $w_3\in I_1$, then $w'\in J_1$, and thus $J=J_1$, which is a contradiction.
	
	(ii) If $B\neq B_1$, then for any $b\in B-B_1$, we have $N(b)\cap ((J-J_1)\cup J')\neq\emptyset$ by the definitions of $B$ and $B_1$. Let $j\in N(b)\cap ((J-J_1)\cup J')$. If $j\in J-J_1$, then by (i), we have $I'_1=I'_3=\emptyset$. If $j\in J'$ and $I'_1\cup I'_3\neq\emptyset$, then each shortest $(w,j)$-path must go through $J$ for any $w\in I'_1\cup I'_3$, and thus $d(w,J)\leq 3$ by $d(G)=4$, a contradiction with the definitions $I'_1$ and $I'_3$. Therefore, $I'_1\cup I'_3=\emptyset$.
\end{proof}

To orient the edges of \([I, K']\) and \([K, K']\), we use the following notations.

For \(w \in K'\), let \(Q_w = w i_1 \cdots i_\ell\) denote a shortest \((w, B)\)-path, where $\ell\in \{3,4\}$ and \(i_\ell \in B\). If \(\ell = 3\), then \(i_1 \in I_1\) and $i_2\in J_1$. If $\ell=4$, then $i_1\in (\cup_{i=2}^4 I_i)\cup K\cup K'$, $i_2\in I_1\cup S_{3,3}\cup L_1$ and $i_3\in \cup_{i=1}^3 J_i$. Now we give a partition of \(K'\) such that $K'=\cup_{i=1}^5 K'_i$, and $K'_i\cap K'_j=\emptyset$ for $1\leq i<j\leq 5$, where
\begin{align*}
	K'_1 &= \{ w \mid w \in K' \text{ and } \forall\  Q_w = w i_1 i_2 i_3 i_4\text{ satisfies } i_1 \in K_2, i_2 \in I_1, i_3\in J_1\text{ and } i_4 \in B \}, \\
	K'_2 &= \{ w \mid w \in K' \text{ and } \exists\  Q_w = w i_1 i_2 i_3i_4 \text{ with } i_1 \in K_1, i_2 \in S_{3,3}\cup L_1,i_3\in \cup_{i=1}^3 J_i, i_4 \in B \}, \\
	K'_3 &= \{ w \mid w \in K'-K'_1-K'_2,\ w \text{ is isolated in } G[K'-K'_1-K'_2] \}, \\
	K'_4 &= \{ w \mid w \in K' - \cup_{i=1}^3 K_i' ,\  d(w, B) = 3 \},\ 
	K'_5 = \{ w \mid w \in K' - \cup_{i=1}^3 K_i',\  d(w, B) = 4 \}.
\end{align*}

\begin{proposition}\label{prop2}
	If $B\neq B_1$, then $K'_1=\emptyset$. 
\end{proposition}

\begin{proof}
	Suppose $K'_1\neq\emptyset$. Let $w\in K'_1$. For any $w'\in B$, by the definition of $K'_1$ and $d(G)=4$, we have $d(w,B)=4$, and thus $d(w,w')=4$ for any $w'\in B$. Let $P=ww_1w_2w_3w'$ be a shortest $(w,w')$-path. Then $P$ is a shortest $(w,B)$-path. By the definition of $K'_1$, we see that $w_1\in K_2$, $w_2\in I_1$ and $w_3\in J_1$, and thus $w'\in B_1$. By the arbitrariness of $w'$, we have $B=B_1$, which is a contradiction.
\end{proof}

Furthermore, we partition $I_8$, $K'_5$, $K'_3$ and $K'_4$ such that $I_8=I_8^{(1)}\cup I_8^{(2)}$, $K'_5=K'_{51}\cup K'_{52}\cup K'_{53}\cup K'_{54}$, $K'_3=K'_{31}\cup K'_{32}$, \(K'_4=K'_{41}\cup K'_{42}\cup K'_{43}\), $I_8^{(1)}\cap I_8^{(2)}=\emptyset$, $K'_{5i}\cap K'_{5j}=\emptyset$ for $1\leq i<j\leq 4$, $K'_{31}\cap K'_{32}=\emptyset$, and $K'_{4i}\cap K'_{4j}=\emptyset$ for $1\leq i<j\leq 3$, where
\begin{align*}
	K'_{51} &= \{ w \mid w \in K'_5 \text{ and } N(w) \cap K'_4 \neq \emptyset \}, \\
	I_8^{(1)}&=\{w\mid w\in I_8, N(w)\cap (K'_2\cup K'_{51})=\emptyset \text{, and for some }w' \in N(w)\cap I'_3,\text{ there exists a }\\&\quad \text{  shortest $(w',J)$-path } P_{w'}=w'wi_2i_3i_4  \text{ such that } i_2\in A_1, i_3\in I_1\text{ and } i_4\in J \},\\
	I_8^{(2)}&=I_8-I_8^{(1)};\\
	K'_{52} &= \{ w \in K'_5 - K'_{51} \mid   w \text{ is isolated in } G[K'_5 - K'_{51}]\text{ or }N(w)\cap \left((\cup_{i=5}^7 I_i)\cup I_8^{(2)}\right)\neq \emptyset \}, \\
	K'_{53} &= \{ w \mid w \in K'_5 - K'_{51}-K'_{52},\  w \text{ is isolated in } G[K'_5 - K'_{51}] \}, \\
	K'_{54} &= \{ w \mid w \in K'_5 - K'_{51}-K'_{52} ,\  w \text{ is not isolated in } G[K'_5 - K'_{51}] \};\\
	K'_{31} &= \{  w \in K'_3 \mid  N(w) \cap \left((\cup_{i=5}^7 I_i)\cup I_8^{(2)}\cup K\right)\neq\emptyset\}\cup\{ w\in K'_3\mid  N(w)\cap K'_1\neq \emptyset ,\  \\&\quad\ \   N(w')\cap (I-I_8^{(1)})\neq\emptyset \text{ for some } w'\in N(w)\cap K'_1 \}, \\
	K'_{32}&= K'_3-K'_{31};\\
	K'_{41} &= \{ w \mid w \in K'_4 ,\ N(w) \cap K'_5 \neq \emptyset \}, \\
	K'_{42} &= \{ w \mid w \in K'_4 - K'_{41} ,\  w \text{ is isolated in } G[K'_4 - K'_{41}] \}, \\
	K'_{43} &= \{ w \mid w \in K'_4 - K'_{41},\  w \text{ is not isolated in } G[K'_4 - K'_{41}] \}.	
\end{align*}

 We note that the sets $\{u\},A,A',I, I', K, K'$ and $\{v\},B, B', J, J', L, L'$ possess symmetry in Figure \ref{fig1}. Then we can partition some sets similarly. For example, by symmetry and using partitioning methods similar to those of \(K'\), \(I'\), and \(I_8\), we partition $L'$, $J'$ and $J_8$ such that $L'=\cup_{i=1}^5 L'_i$, $J'=\cup_{i=1}^8 J'_i$ and $J_8=J_8^{(1)}\cup J_8^{(2)}$.

 Moreover, for all edges between two given sets, the “symmetry” of orientations of these edges means reverse orientations of all edges between the corresponding two sets. For example, in Construction \ref{con1}, \(K \rightarrow I\) corresponds to \(J \rightarrow L\), \((I - I_8) \rightarrow A\) corresponds to \(B \rightarrow (J - J_8)\), and so on.

 In Construction \ref{con1}, we do not give all orientations of $[A,I]$ and $[B,J]$. Now we give all orientations of unoriented edges in $[A,I]$ and $[B,J]$.

\begin{construction}\label{con3}
	Let $A_1\rightarrow I_8^{(1)}$, $J_8^{(1)}\rightarrow B_1$, $I_{8}^{(2)}\rightarrow A$, $B\rightarrow J_8^{(2)}$, $I\rightarrow I'_3\rightarrow A$, and orient other unoriented edges of $[A,I]$ or $[B,J]$ as from $I$ to $A$ or from $B$ to $J$, respectively.
\end{construction}

\begin{claim}\label{claim4}
	{\rm (i)} If $w\in I$, then $\partial(w,u)\leq 3$, with equality only if $w\in I_8^{(1)}$. {\rm (ii)} If $w\in J$, then $\partial(v,w)\leq 3$, with equality only if $w\in J_8^{(1)}$.
\end{claim}

\begin{proof}
	By symmetry, we only consider $w\in I$. By the definition of $I_8^{(1)}$, we have $N(w)\cap I'_3\neq\emptyset$ for $w\in I_8^{(1)}$. For $w\in I$, by Constructions \ref{con1} and \ref{con3}, we have $w\rightarrow A\rightarrow u$ if $w\notin I_8^{(1)}$, and $w\rightarrow I'_3\rightarrow A\rightarrow u$ if $w\in I_8^{(1)}$. Then the result holds.
\end{proof}

For $w\in K'_4\cup (K'_5-K'_{51})$, by checking every shortest $(w,B)$-path, then there exists a shortest $(w,B)$-path $Q_w=wi_1i_2\cdots i_{\ell}$ such that $\ell=3$, $i_1\in I_1$, $i_2\in J_1$ if $w\in K'_4$, and $\ell=4$, $i_1\in \cup_{i=2}^4 I_i$, $i_2\in I_1\cup S_{3,3}$, $i_3\in J_1\cup J_2$ if $w\in K'_5- K'_{51}$. Thus $N(w) \cap I_1 \neq \emptyset$ if $w\in K'_4$, and $N(w) \cap \Big(N(I_1\cup S_{3,3})\cap(\cup_{i=2}^4 I_i)\Big) \neq \emptyset$ if $w \in K'_5-K'_{51}$. Since $G[K'_{43}]$ and $G[K'_{54}]$ have no isolated vertices, we orient some edges from $K'$ as follows.

\begin{construction}\label{con7}
	Let $K'\rightarrow \cup_{i=5}^8 I_i$, $K\rightarrow K'_1\cup K'_2\rightarrow I$, $K'_2\rightarrow K'_1$, $K\rightarrow K'_4$, $\cup_{i=1}^4 I_i\rightarrow K'_{41}\cup K'_{52}\rightarrow K'_{42}\cup K'_{51}\rightarrow I$, $\cup_{i=1}^4 I_i\rightarrow K'_{53}\rightarrow K'_{52}\rightarrow (\cup_{i=5}^7 I_i)\cup I_8^{(2)}$. Let $R = I_1$, $S = K'_{43}$, or $R = N(I_1\cup S_{3,3})\cap(\cup_{i=2}^4 I_i)$, $S = K'_{54}$. Then give an $R$-$S$ orientation for the edges of $G[S]$ and $[R, S]$.
\end{construction}

\begin{claim}\label{claim6}
	Let $w\in K'_1\cup K'_2\cup K'_4\cup K'_5$. Then
	
	 $\partial(w,u)\leq \begin{cases} 3, &\text{if } w\in K'_2,\\  4, &\text{if } w\in K'_1\cup K'_4\cup K'_5;\end{cases}$ and $\partial(v,w)\leq \begin{cases} 5, &\text{if } w\notin K'_1\cup K'_{54},\\  6, &\text{if } w\in K'_1\cup K'_{54}.\end{cases}$
\end{claim}
\begin{proof}
	\vspace{0.1\baselineskip}
	{\noindent \rm \bf Case 1.} $w\in K'_1$.
	\vspace{0.1\baselineskip}
	
	If $w\in K'_1$, then $N(w)\cap K_2\neq\emptyset$ and $N(w)\cap I\neq\emptyset$ by the definition of $K'_1$. Thus $\partial(w,u)\leq \partial(w,w')+\partial(w',u)\leq 1+3=4$ for any $w'\in N(w)\cap I$ by Construction \ref{con7} and Claim \ref{claim4}, and $\partial(v,w)\leq \partial(v,w'')+\partial(w'',w)\leq 5+1=6$ for any $w''\in N(w)\cap K_2$ by (iii) of Claim \ref{claim16} and Construction \ref{con7}.
	
	\vspace{0.1\baselineskip}
	{\noindent \rm \bf Case 2.} $w\in K'_2$.
	\vspace{0.1\baselineskip}
	
	If $w\in K'_2$, then by the definitions of $K'_2$ and $I_8^{(1)}$, we have $N(w)\cap I\neq\emptyset$ and $N(w)\cap I_8^{(1)}=\emptyset$, and thus $N(w)\cap (I-I_8^{(1)})\neq\emptyset$. Therefore, $w\rightarrow (I-I_{8}^{(1)})$ and $\partial(w,u)\leq \partial(w,w')+\partial(w',u)\leq  1+2=3$ for $w'\in N(w)\cap (I-I_8^{(1)})$ by Construction \ref{con7} and Claim \ref{claim4}. On the other hand, by the definition of $K'_2$, there exists $Q_w=wi_1i_2i_3i_4$ such that $i_1 \in K_1, i_2 \in S_{3,3}\cup L_1,i_3\in \cup_{i=1}^3 J_i, i_4 \in B$. Then we have $v\rightarrow i_4\rightarrow i_3\rightarrow i_2\rightarrow i_1\rightarrow w$ by Constructions \ref{con1} and \ref{con7}, and thus $\partial(v,w)\leq 5$.
	
	\vspace{0.1\baselineskip}
	{\noindent \rm \bf Case 3.} $w\in K'_{41}\cup K'_{42}\cup K'_{51}$.
	\vspace{0.1\baselineskip}
	
	By the definitions of $K'_{41}$, $K'_{42}$, $K'_{51}$ and $I_8^{(1)}$, we have $N(w)\cap I_1\neq \emptyset$ if $w\in K'_{41}\cup K'_{42}$, $N(w)\cap K'_{51}\neq\emptyset$ if $w\in K'_{41}$, and $N(w)\cap K'_{41}\neq\emptyset$, $N(w)\cap (I-I_8^{(1)})\neq\emptyset$ if $w\in K'_{51}$. If $w\in K'_{42}$, then $w$ is not isolated in $G[K'_4\cup K'_5]$ by the definition of $K'_3$, and thus $N(w)\cap (K'_4\cup K'_5)\neq\emptyset$, which implies $N(w)\cap K'_{41}\neq\emptyset$ by the definitions of $K'_{41}$ and $K'_{42}$. By Constructions \ref{con1} and \ref{con7}, we have $v\rightarrow B\rightarrow J_1\rightarrow I_1\rightarrow K'_{41}\rightarrow K'_{42}\cup K'_{51}\rightarrow (I-I_8^{(1)})\rightarrow A\rightarrow u$. Thus $\partial(w,u)\leq 4$ and $\partial(v,w)\leq 5$ for any $w\in K'_{41}\cup K'_{42}\cup K'_{51}$.
	
	\vspace{0.1\baselineskip}
	{\noindent \rm \bf Case 4.} $w\in K'_{52}\cup K'_{53}$.
	\vspace{0.1\baselineskip}
	
	If $w\in K'_{52}\cup K'_{53}\subseteq K'_5-K'_{51}$, then by the definition of $K'_5-K'_{51}$, there exists a shortest $(w,B)$-path $Q_w=wi_1i_2i_3i_4$ such that $i_1\in \cup_{i=2}^4 I_i$, $i_2\in I_1\cup S_{3,3}$, $i_3\in J_1\cup J_2$ and $i_4\in B$. Then $v\rightarrow i_4\rightarrow \cdots\rightarrow i_1\rightarrow w$ by Constructions \ref{con1} and \ref{con7}, and thus $\partial(v,w)\leq 5$. 
	
	By the definitions of $K'_{51},K'_{52},K'_{53}$,  Constructions \ref{con1} and \ref{con7}, if $w\in K'_{52}$ and $w$ is isolated in $G[K'_5-K'_{51}]$, then $N(w)\cap K'_{51}\neq\emptyset$, and thus we have $w\rightarrow K'_{51}\rightarrow (I-I_8^{(1)})\rightarrow A\rightarrow u$; if $w\in K'_{52}$ and $N(w)\cap \left((\cup_{i=5}^7 I_i)\cup I_8^{(2)}\right)\neq\emptyset$, then we have $w\rightarrow (\cup_{i=5}^7 I_i)\cup I_8^{(2)}\rightarrow A\rightarrow u$; if $w\in K'_{53}$, then $w$ has a neighbor $w'\in K'_{52}$ and $N(w')\cap \left((\cup_{i=5}^7 I_i)\cup I_8^{(2)}\right)\neq\emptyset$, and thus $w\rightarrow w'\rightarrow (\cup_{i=5}^7 I_i)\cup I_8^{(2)} \rightarrow A\rightarrow u$. Hence $\partial(w,u)\leq 4$ for any $w\in K'_{52}\cup K'_{53}$.
	
	\vspace{0.1\baselineskip}
	{\noindent \rm \bf Case 5.} $w\in K'_{43}\cup K'_{54}$.
	\vspace{0.1\baselineskip}
	
	By the definitions of $K'_{43},K'_{54}$, Constructions \ref{con1}, \ref{con7} and Lemma \ref{l1}, if $w\in K'_{43}$, then we have $v\rightarrow B\rightarrow J_1\rightarrow I_1\rightarrow A\rightarrow u$ and $\theta(w,I_1)\leq 2$, and thus $\partial(w,u)\leq \partial(w,I_1)+\partial(I_1,u)\leq 2+2= 4$, $\partial(v,w)\leq \partial(v,I_1)+\partial(I_1,w)\leq 3+2=5$; if $w\in K'_{54}$, then we have $v\rightarrow B\rightarrow J_1\cup J_2\rightarrow I_1\cup S_{3,3}\rightarrow N(I_1\cup S_{3,3})\cap(\cup_{i=2}^4 I_i) \rightarrow A\rightarrow u$ and $\theta(w,N(I_1\cup S_{3,3})\cap(\cup_{i=2}^4 I_i))\leq 2$, and hence $\partial(w,u)\leq 4$, $\partial(v,w)\leq 6$.
\end{proof}

For any $w\in K'_{32}$, we have $N(w)\cap  \left((\cup_{i=5}^7 I_i)\cup I_8^{(2)}\cup K\right)=\emptyset$ by the definition of $K'_{32}$. By the definition of $K'_3$, we have $N(w)\cap K\subseteq K'_1\cup K'_2$.

\begin{construction}\label{con13}
  {\rm (i)} Let $w\in K'_{31}$. Then we take $\cup_{i=1}^4 I_i\rightarrow w\rightarrow (\cup_{i=5}^7 I_i)\cup I_8^{(2)}\cup K\cup K'_1$. \\  
  {\rm (ii)} Let $w\in K'_{32}$. If $w$ is not an isolated vertex of $G[K']$, then $N(w)\cap (K'_1\cup K'_2)\neq\emptyset$, and thus we take $\cup_{i=1}^4 I_i\rightarrow w\rightarrow K'_2$ if $N(w)\cap K'_2\neq\emptyset$, and $ K'_1\rightarrow w\rightarrow \cup_{i=1}^4 I_i$ if $N(w)\cap K'_2=\emptyset$. If $w$ is an isolated vertex of $G[K']$, then $|[w,I]|\geq 2$ by $N(w)\cap K=\emptyset$, we orient $wi_1$ as $\overrightarrow{i_1w}$ for some shortest $(w,B)$-path $Q_w=wi_1\cdots i_\ell$ with $i_1\in \cup_{i=1}^4 I_i\subseteq I$ and $i_\ell\in B$, and the other edges of $[w,I]$ away from $w$. \\  
  {\rm (iii)} Orient the other undirected edges of $[K,K']$ as $K\rightarrow K'$. 
\end{construction}	

By symmetry and using a similar partitioning method as for \(K'_i\) with $i=3,4,5$, we partition $L'_3, L'_4, L'_5$ into \(L'_{31},L'_{32},L'_{41},L'_{42},L'_{43},L'_{51},L'_{52},L'_{53}, L'_{54}\). By Constructions \ref{con7} and \ref{con13}, we orient some edges incident to \(L'\) similarly, and obtain the following claim, which can be proved similarly to Claim \ref{claim6}.

\begin{claim}\label{claim7}
	Let $w\in L'_1\cup L'_2\cup L'_4\cup L'_5$. Then
	
	$\partial(v,w)\leq \begin{cases} 3, &\text{if } w\in L'_2,\\  4, &\text{if } w\in L'_1\cup L'_4\cup L'_5;\end{cases}$ $\partial(w,u)\leq \begin{cases} 5, &\text{if } w\notin L'_1\cup L'_{54},\\  6, &\text{if } w\in L'_1\cup L'_{54}.\end{cases}$
\end{claim}

\begin{claim}\label{claim18}
	{\rm (i)} For $w\in K'_3$, we have $\partial(w,u)\leq 4$, and
	
	$\partial(v,w)\leq \begin{cases} 7, &\text{if } w\in K'_{32},\ w \text{ is not isolated in } G[K'] \text{ with } N(w)\cap K'_2=\emptyset,\\  5, &\text{otherwise};\end{cases}$
	
	\vspace{5pt}
	{\noindent\rm (ii)} For $w\in L'_3$, we have $\partial(v,w)\leq 4$, and 
	
	$\partial(w,u)\leq \begin{cases} 7, &\text{if } w\in L'_{32},\ w \text{ is not isolated in } G[L'] \text{ with } N(w)\cap L'_2=\emptyset,\\  5, &\text{otherwise}.\end{cases}$
	
\end{claim}
\begin{proof}
	By symmetry, we only consider $w\in K'_3$. Let $w\in K'_3$. Then by the definition of $K'_3$, there exists a shortest $(w,B)$-path $Q_w=wi_1i_2\cdots i_{\ell}$ such that $i_1\in I_1$, $i_2\in J_1$, $i_3\in B_1$ if $\ell=3$, $i_1\in \cup_{i=2}^4 I_i$, $i_2\in I_1\cup S_{3,3}$, $i_3\in J_1\cup J_2$, $i_4\in B_1\cup B_2$ if $\ell=4$. Thus $N(w)\cap (\cup_{i=1}^4 I_i)\neq\emptyset$. In the following, we show the result holds by $K'_3=K'_{31}\cup K'_{32}$.
	
	If $w\in K'_{31}$, then by the definition of $K'_{31}$, we have $N(w)\cap \left((\cup_{i=5}^7 I_i)\cup I_8^{(2)}\cup K\right)\neq\emptyset$, or $N(w)\cap K'_1\neq\emptyset$, $N(w')\cap (I-I_8^{(1)})\neq\emptyset$ for some $w'\in N(w)\cap K'_1$. By Constructions \ref{con1} and \ref{con13}, if $N(w)\cap ((\cup_{i=5}^7 I_i)\cup I_8^{(2)}\cup K)\neq\emptyset$, then $v\rightarrow i_\ell \rightarrow \cdots\rightarrow i_2\rightarrow i_1\rightarrow w\rightarrow (\cup_{i=5}^7 I_i)\cup I_8^{(2)}\cup K$; if $N(w)\cap K'_1\neq\emptyset$ and $N(w')\cap (I-I_8^{(1)})\neq\emptyset$ for some $w'\in N(w)\cap K'_1$, then $v\rightarrow i_\ell \rightarrow \cdots\rightarrow i_2\rightarrow i_1\rightarrow w\rightarrow w'\rightarrow (I-I_8^{(1)})$. Thus $\partial(w,u)\leq 4$ and $\partial(v,w)\leq 5$ by Claims \ref{claim16} and \ref{claim4}.
	
	If $w\in K'_{32}$, then $N(w)\cap  \left((\cup_{i=5}^7 I_i)\cup I_8^{(2)}\cup K\right)=\emptyset$ by the definition of $K'_{32}$. If $w$ is not isolated in $G[K']$ and $N(w)\cap K'_2=\emptyset$, then $N(w)\cap K'_1\neq\emptyset$ and $K'_1\rightarrow w\rightarrow \cup_{i=1}^4 I_i$ by the definition of $K'_3$ and Construction \ref{con13}, and hence $\partial(w,u)\leq 3$ and $\partial(v,w)\leq 7$ by Claims \ref{claim4} and \ref{claim6}; if $w$ is isolated in $G[K']$, or $w$ is not isolated in $G[K']$ with $N(w)\cap K'_2\neq\emptyset$, then $v\rightarrow i_\ell \rightarrow \cdots\rightarrow i_2\rightarrow i_1\rightarrow w\rightarrow w'$ for some $w'\in K'_2\cup I$ by Constructions \ref{con1}, \ref{con13} and a suitable choice of $Q_w$, and thus $\partial(v,w)\leq 5$, and $\partial(w,u)\leq \partial(w,w')+\partial(w',u)\leq  4$ by Claims \ref{claim4} and \ref{claim6}. 
\end{proof}

\begin{proposition}\label{propK'3}
	Let $w\in K'$ and $\widetilde{w}\in L'$. Then\\
	{\rm (i)} if $N(w)\cap I_1\neq\emptyset$, we have $\partial(v,w)\leq 5$; \\
	{\rm (ii)} if $N(w)\cap K_1\neq\emptyset$, we have {\rm (1)} $K\rightarrow w\rightarrow \cup_{i=5}^8 I_i$ and $\partial(v,w)\leq 5$, or {\rm (2)} $\partial(v,w)\leq 4$;\\
	{\rm (iii)} if $N(\widetilde{w})\cap J_1\neq\emptyset$, we have $\partial(\widetilde{w},u)\leq 5$; \\
	{\rm (iv)} if $N(\widetilde{w})\cap L_1\neq\emptyset$, we have {\rm (1)} $\cup_{i=5}^8 J_i\rightarrow \widetilde{w}\rightarrow L$ and $\partial(\widetilde{w},u)\leq 5$, or {\rm (2)} $\partial(\widetilde{w},u)\leq 4$.
\end{proposition}

\begin{proof}
	By symmetry, we only prove (i) and (ii). 
	
	(i) By $N(w)\cap I_1\neq\emptyset$, there exists a shortest $(w,B)$-path $ww_1w_2w_3$, where $w_1\in N(w)\cap I_1$, $w_2\in N(w_1)\cap J$ and $w_3\in N(w_2)\cap B$. Hence $d(w,B)=3$, and thus $w\in K'_3\cup K'_4$ by the definitions of $K'_3$ and $K'_4$.
	
	If $w\in K'_{32}$, and $w$ is not isolated in $G[K']$ with $N(w)\cap K'_2=\emptyset$, then $N(w)\cap K'_1\neq\emptyset$ by the definition of $K'_3$, and thus there exists a shortest $(w',B)$-path $Q_{w'}=w'ww_1w_2w_3$ with $w'\in N(w)\cap K'_1$, which contradicts the definition of $K'_1$. 
	
	Therefore, we have $\partial(v,w)\leq 5$ by Claims \ref{claim6} and \ref{claim18}.
	
	(ii) By $N(w)\cap K_1\neq\emptyset$, there exists a $(w,B)$-path $P=ww_1w_2w_3w_4$, where $w_1\in N(w)\cap K_1$, $w_2\in N(w_1)\cap (S_{3,3}\cup L_1)$, $w_3\in N(w_2)\cap (\cup_{i=1}^3 J_i)$ and $w_4\in N(w_3)\cap B$. If $d(w,B)=4$, then $P$ is a shortest $(w,B)$-path, and thus $w\in K'_2$ by the definition of $K'_2$; if $d(w,B)=3$, then $w\in K'_{31}\cup K'_4$ by $N(w)\cap K\neq\emptyset$, the definitions of $K'_{31}$ and $K'_4$.
	
	If $w\in K'_2\cup K'_4$, then $K\rightarrow w\rightarrow  \cup_{i=5}^8 I_i$ by Construction \ref{con7}, and $\partial(v,w)\leq 5$ by Claim \ref{claim6}. If $w\in K'_{31}$, then $d(w,B)=3$, and thus $N(w)\cap I_1\neq\emptyset$, and $\partial(v,w)\leq \partial(v,w')+\partial(w',w)\leq 3+1=4$ for $w'\in N(w)\cap I_1$ by Claim \ref{claim16} and $I_1\rightarrow w$ in Construction \ref{con13}.
\end{proof}

\begin{claim}\label{claim19}
	Let $w\in I_8$, $\widetilde{w}\in J_8$. Then we have
	
	$\partial(v,w)\leq \begin{cases}5 ,\text{ if }w\in I_8^{(1)},\\ 6,\text{ if  $w\in I_8^{(2)}$ and $K'_1=\emptyset$},\\7,\text{ if  $w\in I_8^{(2)}$ and $K'_1\neq\emptyset$}; \end{cases}$  $\partial(\widetilde{w},u)\leq \begin{cases}5 ,\text{ if }\widetilde{w}\in J_8^{(1)},\\ 6,\text{ if  $\widetilde{w}\in J_8^{(2)}$ and $L'_1=\emptyset$},\\7,\text{ if  $\widetilde{w}\in J_8^{(2)}$ and $L'_1\neq\emptyset$}. \end{cases}$
\end{claim}

\begin{proof}
	By symmetry, we only consider the case of $w\in I_8$. If $w\in I_8^{(1)}$, then by the definition of $I_8^{(1)}$ and Construction \ref{con1}, we have $N(w)\cap A_1\neq\emptyset$ and $v\rightarrow B_1\rightarrow J_1 \rightarrow I_1\rightarrow A_1\rightarrow w$, which implies $\partial(v,w)\leq 5$. If $w\in I_8^{(2)}$, then $N(w)\cap (K'-K'_{32}-K'_{54})\neq\emptyset$ by the definitions of $I_8^{(2)},K'_{32}$ and $K'_{54}$, and thus $w'\rightarrow w$ by Construction \ref{con7}, where $w'\in  N(w)\cap (K'-K'_{32}-K'_{54})$. Therefore, $\partial(v,w)\leq \partial(v,w')+\partial(w',w)\leq \begin{cases}
		6, \text{ if }K'_1=\emptyset\\
		7, \text{ if }K'_1\neq\emptyset
	\end{cases}$ by Claims \ref{claim6} and \ref{claim18}.
\end{proof}

To obtain some directed distances from $I'$, we further partition $I'_7$, $I'_8$, $A_{10}$ and $I'_1$ such that $I'_7=I'_{71}\cup I'_{72}\cup I'_{73}\cup I'_{74}$, \(I'_8=I'_{81}\cup I'_{82}\cup I'_{83}\cup I'_{84}\), $A_{10}=A_{10}^{(1)}\cup A_{10}^{(2)}$, $I'_1=I'_{11}\cup I'_{12}\cup I'_{13}\cup I'_{14}$, $I'_{7i}\cap I'_{7j}=\emptyset$, $I'_{8i}\cap I'_{8j}=\emptyset$, $I'_{1i}\cap I'_{1j}=\emptyset$ for $1\leq i<j\leq 4$, $A_{10}^{(1)}\cap A_{10}^{(2)}=\emptyset$, where
\begin{align*}
	I'_{71} &= \{w \mid w \in I'_7 ,\  N(w) \cap I'_6 \neq \emptyset\}, \ 
	I'_{72} =\{w \mid w \in I'_7 -I'_{71},\  N(w) \cap  I'_8 \neq \emptyset\}, \\
	I'_{73} &= \{w \mid w \in I'_7 - I'_{71}-I'_{72},\  w \text{ is isolated in } G[I'_7 - I'_{71}-I'_{72}]\},\\
	I'_{74} &= \{w \mid w \in I'_7 - I'_{71}-I'_{72},\  w \text{ is not isolated in } G[I'_7 - I'_{71}-I'_{72}]\};\\
	I'_{81} &= \{w \mid w \in I'_8 ,\  N(w) \cap I'_7 \neq \emptyset\}, \\
	I'_{82} &= \{w \in I'_8 - I'_{81} \mid  w \text{ is isolated in } G[I'_8 - I'_{81}] \text{ or }N(w)\cap A_{10}\neq\emptyset\}, \\
	I'_{83} &= \{w \mid w \in I'_8 - I'_{81}-I'_{82} ,\  w \text{ is isolated in } G[I'_8 - I'_{81}-I'_{82}] \}, \\
	I'_{84} &= \{w \mid w \in I'_8 - I'_{81}-I'_{82} ,\  w \text{ is not isolated in } G[I'_8 - I'_{81}-I'_{82}]\};\\
	A_{10}^{(1)}&=\{w\mid w\in A_{10}, \text{ for some }w' \in I'_1, \text{ there exists a shortest $(w',J)$-path } P_{w'} \\&\quad\  \text{ such that }w\in V(P_{w'}) \},\quad
	A_{10}^{(2)}=A_{10}-A_{10}^{(1)};\\
	I'_{11} &= \{w \mid w \in I'_1 ,\  N(w) \cap I'_{73} \neq \emptyset\}, \\
	I'_{12} &= \{w \mid w \in I'_1-I'_{11} ,\  w \text{ is isolated in } G[I'_1-I'_{11}] \text{ or } N(w)\cap A_{10}^{(2)}\neq\emptyset\},\\
	I'_{13} &= \{w \mid w \in I'_1-I'_{11}-I'_{12} ,\  w \text{ is isolated in } G[I'_1-I'_{11}-I'_{12}]\},\\
	I'_{14} &= \{w \mid w \in I'_1-I'_{11}-I'_{12} ,\  w \text{ is not isolated in } G[I'_1-I'_{11}-I'_{12}]\}.
\end{align*}
 

For $w\in I'_1\cup (I'_7-I'_{71})\cup (I'_8-I'_{81})$, by the definitions of $I'_{1},A_{10}^{(1)},I'_7,I'_{71},I'_4,I'_8,I'_{81},I'_1,I'_2$ and $I'_3$, after checking every shortest $(w,J)$-path, we have:
(1) if $w\in I'_{1}$, then there exists a shortest $(w,J)$-path $P_w=wi_1i_2i_3i_4$ such that $i_1\in (\cup_{i=2}^{9}A_i)\cup A_{10}^{(1)}$, $i_2 \in  I'$, $i_3\in I_1$, and we define the set of such $i_1$ as $N^*_1(w)$; 
(2) if $w\in I'_7-I'_{71}$, then $N(w)\cap A_1\neq \emptyset$;
(3) if $w\in I'_8-I'_{81}$, then there exists a shortest $(w,J)$-path $P_w=wi_1i_2i_3i_4$ with $i_1\in \cup_{i=2}^9 A_i$, $ i_2\in A_1\cup (\cup_{i=2}^4 I_i)$, $i_3\in I_1\cup S_{3,3}$, and we define the set of such $i_1$ as $N^*_2(w)$. 

Let $N^*(I'_1)=\mathop{\textstyle\bigcup}\limits_{w\in I'_1} N^*_1(w)$ and $N^*(I'_8-I'_{81})=\mathop{\textstyle\bigcup}\limits_{w\in I'_8-I'_{81}} N^*_2(w)$. Clearly, $N^*(I'_1)\subseteq (\cup_{i=2}^{9}A_i)\cup A_{10}^{(1)}$ and $N^*(I'_8-I'_{81})\subseteq \cup_{i=2}^9 A_i$. Based on the above arguments, we can give Construction \ref{con5} as follows.

\begin{construction}\label{con5}
	{\rm (i)} Let $(I'-I'_1)\rightarrow A_{10}$, $I'_1\rightarrow A_{10}^{(2)}$, $I'_{73}\rightarrow I'_{11}\rightarrow A$, $I\rightarrow \cup_{i=2}^4 I'_i\rightarrow A$, $I'_4\rightarrow I'_2$, $\cup_{i=1}^4 I_i\rightarrow I'_{6}\rightarrow A\cup I'_{71}\cup  (I_7\cup I_8)$, $\cup_{i=1}^4 I_i\rightarrow I'_7\rightarrow A-A_1$, $I'_{71}\rightarrow A$, $A_1\rightarrow I'_{72}\cup I'_{73}$, $I'_{73}\rightarrow I'_{71}\cup I'_{72}\rightarrow I'_{81}\rightarrow A$, $\cup_{i=2}^9 A_i\rightarrow I'_{82}\rightarrow I'_{81}$, $\cup_{i=2}^9 A_i\rightarrow I'_{83}\rightarrow I'_{82}$, $(\cup_{i=2}^{9}A_i)\cup A_{10}^{(1)}\rightarrow I'_{13}\rightarrow I'_{12}$.\\	
	{\rm (ii)} Let $w\in I'_{12}$. If $w$ is not isolated in $G[I\cup I']$, then let $(\cup_{i=2}^{9}A_i)\cup A_{10}^{(1)}\rightarrow I'_{12}\rightarrow I\cup (I'-I'_1-I'_{73})\cup I'_{11}$; if $w$ is isolated in $G[I\cup I']$, then $|[w,A]|\geq 2$ since $G$ is bridgeless, and thus we orient $wi_1$ as $\overrightarrow{i_1w}$ for some shortest $(w,J)$-path $P_w=wi_1i_2i_3i_4$ with $i_1\in (\cup_{i=2}^{9}A_i)\cup A_{10}^{(1)}$, and orient other edges of $[w,A]$ away from $w$. \\	
	{\rm (iii)} Let $(R,S)\in\{\big(N^*(I'_1),I'_{14}\big),\ \big(A_1,I'_{74}\big),\ \big(N^*(I'_8-I'_{81}),I'_{84}\big)\}$. Then we give an $R$-$S$ orientation for the edges of $G[S]$ and $[R,S]$. 
\end{construction}

Moreover, when $(R,S)=\big(N^*(I'_1),I'_{14}\big)$, by Definition \ref{def2}, there exists a partition $V_1(I'_{14}),V_2(I'_{14})$ of $I'_{14}$ such that $R\rightarrow V_1(I'_{14})\rightarrow V_2(I'_{14})\rightarrow R$. Then we obtain Proposition \ref{prop3.5} by Observation \ref{l1.5}.

\begin{proposition}\label{prop3.5}
 Let $I'_{14}=V_1(I'_{14})\cup V_2(I'_{14})$ defined as above. Then 
 
 $\partial(N^*(I'_1),w)= \begin{cases}1 ,\text{ if }w\in V_1(I'_{14}),\\ 2,\text{ if  $w\in V_2(I'_{14})$}, \end{cases}$  $\partial(w,N^*(I'_1))= \begin{cases}2 ,\text{ if }w\in V_1(I'_{14}),\\ 1,\text{ if  $w\in V_2(I'_{14})$}.\end{cases}$
\end{proposition}

\vspace{5pt}
	Let $w\in I'_5$. By the definition of $I'_5$, we have $N(w)\cap I'\subseteq \cup_{i=1}^4 I_i$ and there exists a shortest $(w,J)$-path $P_w=wi_1\cdots i_\ell$ such that $i_1\in I_1$, $i_2\in J_1$ if $\ell=2$; $i_1\in A_1$, $i_2\in I_1$, $i_3\in J_1$ if $\ell=3$; $i_1\in \cup_{i=2}^9 A_i$, $i_2\in A_1\cup (\cup_{i=2}^4 I_i)$, $i_3\in S_{3,3}\cup I_1$, $i_4\in J$ if $\ell=4$. If $w$ is not an isolated vertex in $G[I']$, then $N(w)\cap (\cup_{i=1}^4I'_i)\neq \emptyset$ by the definition of $I'_5$. Specially, if $\ell\in\{2,3\}$, then $N(w)\cap I'_1=\emptyset$. Otherwise, for $w'\in N(w)\cap I'_1$, we have $d(w',J)\leq d(w',w)+d(w,J)=3$ if $\ell=2$, and $w'wi_1i_2i_3$ is a shortest $(w,J)$-path with $w\in I'_5$ if $\ell=3$, both of which contradict the definition of $I'_1$. 
	
	Based on above arguments, we can give Construction \ref{con6} as follows.

\begin{construction}\label{con6}
	{\rm (i)} For $w\in I'_5$, let $P_w=wi_1\cdots i_\ell$ be the shortest $(w,J)$-path with $2\leq \ell \leq 4$ as above. Then we orient the edges incident to $w$ as follows.
		
	If $w$ is not an isolated vertex in $G[I']$ and $\ell\in\{2,3\}$, we take $(\cup_{i=1}^4 I_i)\cup A_1\rightarrow w\rightarrow (\cup_{i=2}^4 I'_i)\cup(A-A_1)\cup I_7\cup I_8$.
	
	If $w$ is not an isolated vertex in $G[I']$ and $\ell=4$, we take:
	 {\rm (1)} if $N(w)\cap \Big(A_{10}\cup(\cup_{i=2}^4I'_i)\cup I'_{11}\cup V_2(I'_{14})\Big)\neq \emptyset$, then $\cup_{i=2}^9 A_i\rightarrow w\rightarrow A_{10}\cup (\cup_{i=2}^4I'_i)\cup I'_{11}\cup V_2(I'_{14})$;
	 {\rm (2)} if $N(w)\cap \Big(A_{10}\cup(\cup_{i=2}^4I'_i)\cup I'_{11}\cup V_2(I'_{14})\Big)=\emptyset$, then $I'_{12}\cup I'_{13}\cup V_1(I'_{14})\rightarrow w\rightarrow A$.
	 
	 If $w$ is an isolated vertex in $G[I']$ and $N(w)\cap  I=\emptyset$, then $|[w,A]|\geq 2$ and $\ell\in\{3,4\}$, and thus we orient $wi_1$ as $\overrightarrow{i_1w}$ for the shortest $(w,J)$-path $P_w=wi_1i_2\cdots i_\ell$ with $\ell\in \{3,4\}$, and the other edges of $[w,A]$ away from $w$.
	 
	 If $w$ is an isolated vertex in $G[I']$ and $N(w)\cap I\neq\emptyset$, then we take:
	 {\rm (1)} if $\ell=2$, then $\cup_{i=1}^4 I_i\to w\to A\cup I_7\cup I_8$;
	 {\rm (2)} if $\ell=3$, then $(\cup_{i=1}^4 I_i)\cup I_8^{(1)}\rightarrow w\rightarrow A$ for $N(w)\cap((\cup_{i=1}^4 I_i)\cup I_8^{(1)})\neq\emptyset$, and $A_1\rightarrow w\rightarrow (\cup_{i=5}^7 I_i)\cup I_8^{(2)}$ for $N(w)\cap ((\cup_{i=1}^4 I_i)\cup I_8^{(1)})=\emptyset$;
	 {\rm (3)} if $\ell=4$, then $I_8^{(1)}\to w\to A$ for $N(w)\cap I_8^{(1)}\neq\emptyset$, and $\cup_{i=2}^{9}A_i\to w\to(I- I_8^{(1)})$ for $N(w)\cap I_8^{(1)}=\emptyset$.

	 {\rm (ii)} Orient all other undirected edges of $[I,I']$ from $I$ to $I'$. 
	\end{construction}

By symmetry and using partitioning methods similar to those of  \(A_{10}\) and \(I'_i\) for $i=1,7,8$, we partition \(B_{10},  J'_1, J'_7, J'_8\) into \(B_{10}^{(1)},B_{10}^{(2)}\), $J'_{11},J'_{12},J'_{13},J'_{14}$,  $J'_{71},J'_{72},J'_{73},J'_{74}$,  $J'_{81},J'_{82},J'_{83},J'_{84}$, respectively. By Constructions \ref{con5}, \ref{con6} and symmetry, we orient some edges incident to \(J'\) similarly. For other undirected edges of $G$, we orient these edges arbitrarily. At this point, we obtain an orientation $\overrightarrow{G}$ of $G$.

\begin{proposition}\label{prop1.5}
	Let $w\in I'$ and $\widetilde{w}\in J'$. Then\\	
	{\rm (i)} if $N(w)\cap I_1\neq\emptyset$, we have $\cup_{i=1}^4 I_i\rightarrow w\rightarrow (A-A_1)\cup I_7\cup I_8$;\\
	{\rm (ii)} if $N(w)\cap (\cup_{i=1}^4 I_i)\neq\emptyset$, we have $\cup_{i=1}^4 I_i\rightarrow w\rightarrow A-A_1$; \\
	{\rm (iii)} if $N(\widetilde{w})\cap J_1\neq\emptyset$, we have $(B-B_1)\cup J_7\cup J_8 \rightarrow \widetilde{w}\rightarrow \cup_{i=1}^4 J_i$;\\
	{\rm (iv)} if $N(\widetilde{w})\cap (\cup_{i=1}^4 J_i)\neq\emptyset$, we have $B-B_1\rightarrow \widetilde{w}\rightarrow\cup_{i=1}^4 J_i$.  
\end{proposition}

\begin{proof}
	By symmetry, we only prove (i) and (ii). Firstly, we have some conclusions as follows: If $w\in I'_6$, then $\cup_{i=1}^4 I_i\rightarrow w\rightarrow (A-A_1)\cup I_7\cup I_8$ by Construction \ref{con5}. If $w\in I'_5$ with $d(w,J)=2$, then $\cup_{i=1}^4 I_i\rightarrow w\rightarrow (A-A_1)\cup I_7\cup I_8$ by Construction \ref{con6} and $N(w)\cap I\neq\emptyset$. If $w\in I'_5$ with $d(w,J)=3$ and $N(w)\cap (\cup_{i=1}^4 I_i)\neq\emptyset$, then $\cup_{i=1}^4 I_i\rightarrow w\rightarrow A-A_1$ by Construction \ref{con6}. If $w\in I'_2\cup I'_4\cup I'_7$, then $\cup_{i=1}^4 I_i\rightarrow w\rightarrow A-A_1$ by Construction \ref{con5}. Secondly, we complete the proof by the following three cases.
	
	\vspace{0.1\baselineskip}
	{\noindent \rm \bf Case 1.} $N(w)\cap I_1\neq\emptyset$.
	\vspace{0.1\baselineskip}
	
	In this case, there exists a shortest $(w,J)$-path $wx_1x_2$, where $x_1\in N(w)\cap I_1$ and $x_2\in N(x_1)\cap J$, and thus $d(w,J)=2$, which implies $w\in I'_5\cup I'_6$ by the definitions of $I'_5$ and $I'_6$. By above arguments, $\cup_{i=1}^4 I_i\rightarrow w\rightarrow (A-A_1)\cup I_7\cup I_8$.
	
	\vspace{0.1\baselineskip}
	{\noindent \rm \bf Case 2.} $N(w)\cap (I_2\cup I_4)\neq\emptyset$.
	\vspace{0.1\baselineskip}
	
	In this case, there exists a $(w,J)$-path $P_1=wy_1y_2y_3$, where $y_1\in N(w)\cap (I_2\cup I_4)$, $y_2\in N(y_1)\cap (S_{3,3}\cup I_1)$ and $y_3\in N(y_2)\cap J$, and thus $d(w,J)\leq 3$. If $d(w,J)=3$, then $P_1$ is a shortest $(w,J)$-path, and thus $w\in I'_4$ by the definition of $I'_4$; if $d(w,J)=2$, then $w\in I'_5\cup I'_6$. By above arguments, $\cup_{i=1}^4 I_i\rightarrow w\rightarrow A-A_1$. 
	
	\vspace{0.1\baselineskip}
	{\noindent \rm \bf Case 3.} $N(w)\cap I_3\neq\emptyset$.
	\vspace{0.1\baselineskip}
	
	In this case, there exists a $(w,J)$-path $P_2=wz_1z_2z_3z_4$, where $z_1\in N(w)\cap I_3$, $z_2\in N(z_1)\cap K_1$, $z_3\in N(z_2)\cap (S_{3,3}\cup L_1)$ and $z_4\in N(z_3)\cap J$, and thus $d(w,J)\leq 4$. If $d(w,J)=4$, then $P_2$ is a shortest $(w,J)$-path, and thus $w\in I'_2$; if $d(w,J)=3$, then $w\in I'_4\cup I'_5\cup I'_7$; if $d(w,J)=2$, then $w\in I'_5\cup I'_6$. By above arguments, $\cup_{i=1}^4 I_i\rightarrow w\rightarrow A-A_1$.	
	
	By the results of Cases 1-3, Case 1 shows (i), and together Cases 1-3 show (ii).
\end{proof}

\begin{claim}\label{claim11}
	If $w\in I'-I'_1-I'_3-I'_5$, then $\partial(w,u)\leq 4$, with equality only if $w\in I'_{73}$; $\partial(v,w)\leq 7$, with equality only if $w\in I'_{84}$. Particularly, if $w\in I'_{73}$, then $\partial(v,w)\leq 5$.
\end{claim}
\begin{proof}	
	We complete the proof by the following four cases.
	
	\vspace{0.1\baselineskip}
	{\noindent \rm \bf Case 1.} $w\in I'_2\cup I'_4$.
	\vspace{0.1\baselineskip}
	
	By Constructions \ref{con1} and \ref{con5}, we have $w\rightarrow A\rightarrow u$, and thus $\partial(w,u)\leq 2$.
	
	By the definitions of $ I'_2$ and $ I'_4$, there exists a $(w,J)$-path $P_w=wi_1\cdots i_\ell$ with $\ell \in \{3,4\}$ and $i_\ell \in \cup_{i=1}^3 J_i$ such that if $w\in I'_4$, then (a) holds, where (a): $i_1\in \cup_{i=2}^4 I_i$, $i_2\in I_1\cup S_{3,3}$; if $w\in I'_2$, then	one of (b), (c), (d), (e) holds, where (b): $i_1\in I_3$, $i_2\in K_1$, $i_3\in S_{3,3}\cup L_1$, (c): $i_1\in \cup_{i=5}^8 I_i$, $i_2\in \cup_{i=2}^4I_i$, $i_3\in S_{3,3}\cup I_1$, (d): $i_1\in I_7\cup I_8$, $i_2\in I'$, $i_3\in I_1$, and (e): $i_1\in I'$, $i_2\in \cup_{i=2}^4I_i$, $i_3\in S_{3,3}\cup I_1$. 
	
	For (a), (b) and (c), we have $v\rightarrow B\rightarrow i_\ell\rightarrow \cdots \rightarrow i_2\rightarrow i_1\rightarrow w$ by Constructions \ref{con1} and \ref{con5}, which implies $\partial(v,w)\leq 6$. For (d), we have $N(i_2)\cap I_1\neq\emptyset$, which implies $I_1\rightarrow i_2\rightarrow i_1$ by (i) of Proposition \ref{prop1.5}; for (e), we have $w\in I'_2$ and $i_1\in I'_4$ by the definition of $I'_4$ and $d(i_1,J)=3$. Therefore, for both (d) and (e), we have $v\rightarrow B\rightarrow i_{4} \rightarrow i_3\rightarrow i_2\rightarrow i_1\rightarrow w$ by Constructions \ref{con1} and \ref{con5}, which implies $\partial(v,w)\leq 6$.
	
	\vspace{0.1\baselineskip}
	{\noindent \rm \bf Case 2.} $w\in I'_6\cup I'_{71}\cup I'_{72}\cup I'_{73}$.
	\vspace{0.1\baselineskip}
	
	 For $w\in I'_{6}\cup I'_{71}$, if $w\in I'_{6}$, then $N(w)\cap I_1\neq\emptyset$ by the definition of $I'_{6}$; if $w\in I'_{71}$, then $N(w)\cap I'_{6}\neq\emptyset$ by the definition of $I'_{71}$. By Constructions \ref{con1} and \ref{con5}, we have $v\rightarrow B\rightarrow J_1\rightarrow I_1\rightarrow I'_{6}\rightarrow I'_{71}$ and $I'_6\cup I'_{71}\rightarrow A\rightarrow u$, which implies $\partial(w,u)\leq 2$ and $\partial(v,w)\leq 5$.

	For $w\in I'_{72}\cup I'_{73}\subseteq I'_7-I'_{71}$, by the definitions of $I'_7$ and $I'_{71}$, there exists a shortest $(w,J)$-path $P_w=wi_1i_2i_3$ such that $i_1\in A_1$, $i_2\in I_1$ and $i_3\in J_1$. Then $v\rightarrow B\rightarrow i_3\rightarrow i_2\rightarrow i_1\rightarrow w$ by Constructions \ref{con1} and \ref{con5}, which implies $\partial(v,w)\leq 5$.
	
   On the other hand, if $w\in I'_{72}$, then $N(w)\cap I'_{81}\neq\emptyset$, and thus we have $w\rightarrow I'_{81}\rightarrow A\rightarrow u$ by Constructions \ref{con1} and \ref{con5}, which implies $\partial(w,u)\leq 3$; if $w\in I'_{73}$, then $N(w)\cap (I'_6\cup I'_7\cup I'_8)\neq \emptyset$ by the definitions of $I'_5$ and $I'_7$, which implies $N(w)\cap (I'_{71}\cup I'_{72})\neq \emptyset$ by the definitions of $I'_{71},I'_{72}$ and $I'_{73}$, and thus $w\rightarrow I'_{71}\rightarrow A\rightarrow u$ if $N(w)\cap I'_{71}\neq\emptyset$, and $w\rightarrow I'_{72}\rightarrow I'_{81}\rightarrow A\rightarrow u$ if $N(w)\cap I'_{72}\neq\emptyset$ by Constructions \ref{con1} and \ref{con5}, which implies $\partial(w,u)\leq 4$. 
	
	\vspace{0.1\baselineskip}
	{\noindent \rm \bf Case 3.} $w\in I'_{81}\cup I'_{82}\cup I'_{83}$.
	\vspace{0.1\baselineskip}
	
	If $w\in I'_{81}$, then $N(w)\cap (I'_{71}\cup I'_{72})\neq\emptyset$ by the definition of $I'_{81}$, and $I'_{71}\cup I'_{72}\rightarrow w\rightarrow A\rightarrow u$ by Constructions \ref{con1} and \ref{con5}, which implies $\partial(w,u)\leq 2$ and $\partial(v,w)\leq \partial(v,w')+\partial(w',w)\leq 5+1=6$ by Case 2, where $w'\in N(w)\cap  (I'_{71}\cup I'_{72})$.
	
	If $w\in I'_{82}\cup I'_{83}\subseteq I'_8-I'_{81}$, then there exists a shortest $(w,J)$-path $P_w=wi_1i_2i_3i_4$ such that $i_1\in \cup_{i=2}^9 A_i$, $i_2\in A_1\cup (\cup_{i=2}^4 I_i)$, $i_3\in I_1\cup S_{3,3}$ and $i_4\in J_1\cup J_2$ by the definitions of $I'_{8}$ and $ I'_{81}$, and thus $v\rightarrow B\rightarrow i_4\rightarrow i_3\rightarrow i_2\rightarrow i_1\rightarrow w$ by Constructions \ref{con1} and \ref{con5}, which implies $\partial(v,w)\leq 6$.
	
	On the other hand, if $w\in I'_{82}\cup I'_{83}$, we show $N(w)\cap I'_8\neq\emptyset$. Firstly, we have $N(w)\cap (I'_6\cup I'_7\cup I'_8)\neq\emptyset$ by the definition of $I'_5$. Secondly, we show $N(w)\cap (I'_6\cup I'_7)=\emptyset$. Otherwise, we have $d(w,J)\leq d(w,w')+d(w',J)=1+2= 3$ for $w'\in N(w)\cap I'_6$ if $N(w)\cap I'_6\neq\emptyset$, and $w\in I'_{81}$ if $N(w)\cap I'_7\neq\emptyset$, both of which contradict $w\in I'_{82}\cup I'_{83}$.
	
	Now we show $\partial(w,u)\leq 3$ for $w\in I'_{82}\cup I'_{83}$. By the definitions of $I'_{82},I'_{83}$, Constructions \ref{con1} and \ref{con5}, if $w\in I'_{82}$ and $w$ is isolated in $G[I'_8-I'_{81}]$, then $N(w)\cap I'_{81}\neq\emptyset$ by $N(w)\cap I'_8\neq\emptyset$, and hence $w\rightarrow I'_{81}\rightarrow A\rightarrow u$, which implies $\partial(w,u)\leq 3$; if $w\in I'_{82}$ and $N(w)\cap A_{10}\neq \emptyset$, then $w\rightarrow A_{10}\rightarrow u$, which implies $\partial(w,u)\leq 2$; if $w\in I'_{83}$, then $w$ has a neighbor $w'\in I'_{82}$ satisfying $N(w')\cap A_{10}\neq\emptyset$, and thus we have $w\rightarrow w'\rightarrow A_{10}\rightarrow u$, which implies $\partial(w,u)\leq 3$.
	
	\vspace{0.1\baselineskip}
	{\noindent \rm \bf Case 4.} $w\in I'_{74}\cup I'_{84}$.
	\vspace{0.1\baselineskip}
 	
 	By Constructions \ref{con1}, \ref{con5} and Lemma \ref{l1}, if $w\in I'_{74}$, then $v\rightarrow B\rightarrow J_1\rightarrow I_1\rightarrow A_1\rightarrow u$ and $\theta(w,A_1)\leq 2$, which implies $\partial(w,u)\leq \partial(w,A_1)+\partial(A_1,u)\leq 2+1= 3$ and $\partial(v,w)\leq \partial(v,A_1)+\partial(A_1,w)\leq 4+2= 6$; if $w\in I'_{84}$, then $v\rightarrow B\rightarrow J_1\cup J_2\rightarrow I_1\cup S_{3,3}\rightarrow A_1\cup (\cup_{i=2}^4 I_i)\rightarrow N^*(I'_8-I'_{81})\rightarrow u$ by $N^*(I'_8-I'_{81})\subseteq \cup_{i=2}^9 A_i$, and $\theta(w,N^*(I'_8-I'_{81}))\leq 2$, which implies $\partial(w,u)\leq 3$ and $\partial(v,w)\leq 7$.
\end{proof}


\begin{claim}\label{claim I'_1}
	Let $w\in I'_1\cup I'_3\cup I'_5$. Then we have\\	
	{\rm (i)} $\partial(w,u)\leq 4$, with equality only if $w\in I'_{12}$; \\{\rm (ii)} $\partial(v,w)\leq 7$, with equality only if $w\in I'_{14}$, or $w\in I'_3$ with $N(w)\cap I\subseteq I_7\cup I_8$, or $w\in I'_5$ with $N(w)\cap A_{10}=\emptyset$ and $d(w,J)=4$.
\end{claim}

\begin{proof} We complete the proof by the following five cases.
	
	\vspace{0.1\baselineskip}
	{\noindent \rm \bf Case 1.} $w\in I'_3$.
	\vspace{0.1\baselineskip}
	
	By Constructions \ref{con1} and \ref{con5}, we have $w\rightarrow A\rightarrow u$, which implies $\partial(w,u)\leq 2$. 
	
  Now we show $\partial(v,w)\leq 7$, with equality only if $N(w)\cap I\subseteq I_7\cup I_8$. By the definition of $I'_3$, we have $N(w)\cap (I_7\cup I_8)\neq\emptyset$.	If $N(w)\cap (I-I_7- I_8)\neq\emptyset$, then $\partial(v,w)\leq \partial(v,w')+\partial(w',w)\leq  5+1=6$ for $w'\in N(w)\cap (I-I_7-I_8) $ by (v) of Claim \ref{claim16} and Construction \ref{con5}. Now we assume $N(w)\cap (I-I_7- I_8)=\emptyset$, i.e., $N(w)\cap I\subseteq I_7\cup I_8$.
	
	If $N(w)\cap (I_7\cup I_8^{(1)})\neq\emptyset$, then $\partial(v,w)\leq \partial(v,w')+\partial(w',w)\leq 6+1= 7$ for $w'\in N(w)\cap (I_7\cup I_8^{(1)}) $ by Claims \ref{claim16}, \ref{claim19} and Construction \ref{con5}. 
	
	If $N(w)\cap (I_7\cup I_8^{(1)})=\emptyset$, i.e., $N(w)\cap I\subseteq I_8^{(2)}\subseteq I_8$, then there exists a shortest $(w,J)$-path $P_w=wi_1i_2i_3i_4$ such that $i_1 \in I_8^{(2)}$, $i_2 \in A_1\cup K'$, $ i_3\in I_1$ and $i_4\in J_1$ by the definitions of $I'_{3}$ and $I_8$. 
	
	Now we consider $N(i_1)$. If $N(i_1)\cap (K'_2\cup K'_{51})\neq\emptyset$, then $K'_2\cup K'_{51}\rightarrow i_1\rightarrow w$ by Constructions \ref{con7} and \ref{con5}, which implies $\partial(v,w)\leq \partial(v,w')+\partial(w',w)\leq 5+2=7$ for $w'\in N(i_1)\cap (K'_2\cup K'_{51})$ by Claim \ref{claim6}. If $N(i_1)\cap (K'_2\cup K'_{51})=\emptyset$, we now study $i_2$. Firstly, we show $i_2\in K'$. Otherwise, $i_2\in A_1$ implies $i_1\in I_8^{(1)}$, a contradiction with $i_1\in I_8^{(2)}$. Secondly, we have $\partial(v,i_2)\leq 5$ by $N(i_2)\cap I_1\neq\emptyset$ and (i) of Proposition \ref{propK'3}. Therefore, we have $i_2\rightarrow i_1\rightarrow w$ by Constructions \ref{con7} and \ref{con5}, which implies $\partial(v,w)\leq \partial(v,i_2)+\partial(i_2,w)\leq 7$. 
	
	\vspace{0.1\baselineskip}
	{\noindent \rm \bf Case 2.} $w\in I'_{11}\cup I'_{14}$.
	\vspace{0.1\baselineskip}
	
	If $w\in I'_{11}$, then there exists $w'\in N(w)\cap I'_{73}$ and $w'\rightarrow w\rightarrow A\rightarrow u$  by the definition of $I'_{11}$, Constructions \ref{con1} and \ref{con5}, which implies $\partial(w,u)\leq 2$ and $\partial(v,w)\leq \partial(v,w')+\partial(w',w)\leq 5+1= 6$ by Claim \ref{claim11}. 
	
	If $w\in I'_{14}$, then $v\rightarrow B\rightarrow J_1\rightarrow I_1\rightarrow N(I_1)\cap I'\rightarrow N^*(I'_1)\rightarrow u$ since Construction \ref{con1} and $I_1\rightarrow w'\rightarrow N^*(I'_1)$ for any $w'\in N(I_1)\cap I'$ by $N(w')\cap I_1\neq\emptyset$, (i) of Proposition \ref{prop1.5} and $N^*(I'_1)\subseteq (\cup_{i=2}^9 A_i)\cup A_{10}^{(1)}$. Thus $\partial(w,u)\leq \partial(w,N^*(I'_1))+\partial(N^*(I'_1),u)\leq 2+1=3$, $\partial(v,w)\leq \partial(v,N^*(I'_1))+\partial(N^*(I'_1),w)\leq 5+1= 6$ if $w\in V_1(I'_{14})$, and $\partial(w,u)\leq 2$, $\partial(v,w)\leq 7$ if $w\in V_2(I'_{14})$ by Proposition \ref{prop3.5}.
	
	\vspace{0.1\baselineskip}
	{\noindent \rm \bf Case 3.} $w\in I'_{12}\cup I'_{13}$ with the upper bound of $\partial(v,w)$.
	\vspace{0.1\baselineskip}
	
	 By the definitions of $I'_{1},A_{10}^{(1)}$ and Construction \ref{con5}, there exists a shortest $(w,J)$-path $P_w=wi_1i_2i_3i_4$ such that $i_1\in (\cup_{i=2}^{9}A_i)\cup A_{10}^{(1)}$, $i_2 \in  I'$, $i_3\in I_1$, $i_4\in J_1$, and $wi_1$ is oriented as $\overrightarrow{i_1w}$. Since $N(i_2)\cap I_1\neq\emptyset$, we have $i_3\rightarrow i_2\rightarrow i_1$ by (i) of Proposition \ref{prop1.5}, and then $v\rightarrow B\rightarrow i_4\rightarrow i_3\rightarrow i_2\rightarrow i_1\rightarrow w$ by Construction \ref{con1}. Thus $\partial(v,w)\leq 6$ for $w\in I'_{12}\cup I'_{13}$.
	
	\vspace{0.1\baselineskip}
	{\noindent \rm \bf Case 4.} $w\in I'_5$.
	\vspace{0.1\baselineskip}
	
	By the definition of $I'_5$, there exists a shortest $(w,J)$-path $P_w=wi_1\cdots i_\ell$ such that $i_1\in I_1$, $i_2\in J_1$ if $\ell=2$; $i_1\in A_1$, $i_2\in I_1$, $i_3\in J_1$ if $\ell=3$; $i_1\in \cup_{i=2}^9 A_i$, $i_2\in A_1\cup (\cup_{i=2}^4 I_i)$, $i_3\in S_{3,3}\cup I_1$, $i_4\in J$ if $\ell=4$. Now we consider the following three cases.
	
	If $w$ is not an isolated vertex in $G[I']$, then $N(w)\cap (\cup_{i=1}^4I'_i)\neq \emptyset$ by the definition of $I'_5$. For $\ell\in\{2,3\}$, we have shown earlier that $N(w)\cap I'_1=\emptyset$, then $v\rightarrow B\rightarrow i_\ell \rightarrow\cdots\rightarrow i_1\rightarrow w\rightarrow \cup_{i=2}^4I'_i\rightarrow A\rightarrow u$ by Constructions \ref{con1}, \ref{con5} and \ref{con6}, which implies $\partial(w,u)\leq 3$ and $\partial(v,w)\leq 5$. For $\ell=4$, if $N(w)\cap \Big(A_{10}\cup(\cup_{i=2}^4I'_i)\cup I'_{11}\cup V_2(I'_{14})\Big)\neq \emptyset$, then $v\rightarrow B\rightarrow i_4\rightarrow\cdots\rightarrow i_1\rightarrow w\rightarrow A_{10}\cup (\cup_{i=2}^4I'_i)\cup I'_{11}\cup V_2(I'_{14})$, $(\cup_{i=2}^4I'_i)\cup I'_{11}\rightarrow A\rightarrow u$ by Constructions \ref{con1}, \ref{con5} and \ref{con6}, and $\partial(w',u)\leq 2$ for $w'\in V_2(I'_{14})$ by Case 2; if $N(w)\cap \Big(A_{10}\cup (\cup_{i=2}^4I'_i)\cup I'_{11}\cup V_2(I'_{14})\Big)=\emptyset$, then $N(w)\cap (I'_{12}\cup I'_{13}\cup V_1(I'_{14}))\neq\emptyset$ by $N(w)\cap (\cup_{i=1}^4I'_i)\neq \emptyset$, and thus $I'_{12}\cup I'_{13}\cup V_1(I'_{14})\rightarrow w\rightarrow A\rightarrow u$ by Constructions \ref{con1} and \ref{con6}, and $\partial(v,w'')\leq 6$ for $w''\in I'_{12}\cup I'_{13}\cup V_1(I'_{14})$ by Cases 2 and 3. Based on the above arguments, we have $\partial(w,u)\leq 3$ and $\partial(v,w)\leq 7$, the latter with equality only if $N(w)\cap A_{10}=\emptyset$ and $\ell=4$.
	
	If $w$ is an isolated vertex in $G[I']$ and $N(w)\cap I=\emptyset$, then $|[w,A]|\geq 2$, $\ell\in \{3,4\}$, and $v\rightarrow B\rightarrow i_\ell\rightarrow \cdots\rightarrow  i_1\rightarrow w$ by Constructions \ref{con1}, \ref{con6} and a suitable choice of $P_w$, which implies $\partial(v,w)\leq 6$. On the other hand, we have $w\rightarrow w'\rightarrow u$ for some $w'\in A$ by Constructions \ref{con1} and \ref{con6}, which implies $\partial(w,u)\leq 2$. 
	
	If $w$ is an isolated vertex in $G[I']$ and $N(w)\cap I\neq\emptyset$, then by Constructions \ref{con1}, \ref{con3} and \ref{con6}, we have: (1) if $\ell=2$, then $v\rightarrow B\rightarrow i_2\rightarrow i_1\rightarrow w\rightarrow A\rightarrow u$, which implies $\partial(w,u)\leq 2$ and $\partial(v,w)\leq 4$; {\rm (2)} if $\ell=3$ and $N(w)\cap((\cup_{i=1}^4 I_i)\cup I_8^{(1)})\neq\emptyset$, then $(\cup_{i=1}^4 I_i)\cup I_8^{(1)}\rightarrow w\rightarrow A\rightarrow u$, which implies $\partial(w,u)\leq 2$ and $\partial(v,w)\leq \partial(v,w_1)+\partial(w_1,w)\leq 5+1=6$ with $w_1\in N(w)\cap \big((\cup_{i=1}^4 I_i)\cup I_8^{(1)}\big)$ by Claims \ref{claim16} and \ref{claim19}; {\rm (3)} if $\ell=3$ and $N(w)\cap ((\cup_{i=1}^4 I_i)\cup I_8^{(1)})=\emptyset$, then $N(w)\cap ((\cup_{i=5}^7 I_i)\cup I_8^{(2)})\neq\emptyset $ by $N(w)\cap I\neq\emptyset$, and thus $v\rightarrow B\rightarrow i_3\rightarrow i_2\rightarrow i_1\rightarrow w\rightarrow (\cup_{i=5}^7 I_i)\cup I_8^{(2)}\rightarrow A\rightarrow u$, which implies $\partial(w,u)\leq 3$ and $\partial(v,w)\leq 5$;	(4) if $\ell=4$ and $N(w)\cap I_8^{(1)}\neq\emptyset$, then $I_8^{(1)}\rightarrow w\rightarrow A\rightarrow u$, which implies $\partial(w,u)\leq 2$ and $\partial(v,w)\leq \partial(v,w_2)+\partial(w_2,w)\leq 5+1=6$ with $w_2\in N(w)\cap I_8^{(1)}$ by Claim \ref{claim19}; (5) if $\ell=4$ and $N(w)\cap I_8^{(1)}=\emptyset$, then $v\rightarrow B\rightarrow i_4\rightarrow i_3\rightarrow i_2\rightarrow i_1\rightarrow w\rightarrow (I-I_8^{(1)})\rightarrow A\rightarrow u$, which implies $\partial(w,u)\leq 3$ and $\partial(v,w)\leq 6$.
	
	\vspace{0.1\baselineskip}
	{\noindent \rm \bf Case 5.} $w\in I'_{12}\cup I'_{13}$ with the upper bound of $\partial(w,u)$.
	\vspace{0.1\baselineskip}
	
	For $w\in I'_{12}$, if $N(w)\cap A_{10}^{(2)}\neq\emptyset$, then $w\rightarrow A_{10}^{(2)}\rightarrow u$ by Constructions \ref{con1} and \ref{con5}, and thus $\partial(w,u)\leq 2$. Otherwise, $w$ is isolated in $G[I'_1-I'_{11}]$ by the definition of $I'_{12}$. Now we consider whether $w$ is isolated in $G[I\cup I']$.
	
	If $w$ is not isolated in $G[I\cup I']$, then there exists $w'\in N(w)\cap \left(I\cup (I'-I'_1-I'_{73})\cup I'_{11}\right)$ by the definition of $I'_{11}$, and thus $w\rightarrow w'$ by Construction \ref{con5}. Furthermore, $\partial(w',u)\leq 3$ by Claims \ref{claim4} and \ref{claim11}, Cases 1, 2 and 4, and then  $\partial(w,u)\leq \partial(w,w')+\partial(w',u)\leq 4$. 
	
	If $w$ is isolated in $G[I\cup I']$, then $w\rightarrow w''\rightarrow u$ for some $w''\in A$ by Constructions \ref{con1} and \ref{con5}, which implies $\partial(w,u)\leq 2$.
	
	For $w\in I'_{13}$, $w$ has a neighbor $w'\in I'_{12}$ with $N(w')\cap A_{10}^{(2)}\neq\emptyset$ by the definition of $I'_{13}$, and thus $w\rightarrow w'\rightarrow A_{10}^{(2)}\rightarrow u$ by Construction \ref{con5}, which implies $\partial(w,u)\leq 3$.
\end{proof}

Similar to the proofs of Claims \ref{claim11} and \ref{claim I'_1}, we obtain Claim \ref{claim12} and omit the proof.

\begin{claim}\label{claim12}
	Let $w\in J'-J'_1-J'_3-J'_5$ and $\widetilde{w}\in J'_1\cup J'_3\cup J'_5$. Then we have\\	
	{\rm (i)} $\partial(v,w)\leq 4$, with equality only if $w\in J'_{73}$; $\partial(w,u)\leq 7$, with equality only if $w\in J'_{84}$; moreover, $\partial(w,u)\leq 5$ if $w\in J'_{73}$;\\{\rm (ii)} $\partial(v,\widetilde{w})\leq 4$, with equality only if $\widetilde{w}\in J'_{12}$; $\partial(\widetilde{w},u)\leq 7$, with equality only if $\widetilde{w}\in J'_{14}$, or $\widetilde{w}\in J'_3$ with $N(\widetilde{w})\cap J\subseteq J_7\cup J_8$, or $\widetilde{w}\in J'_5$ with $N(\widetilde{w})\cap B_{10}=\emptyset$ and $d(\widetilde{w},I)=4$.
\end{claim}

	\begin{claim}\label{claim20}
	Let $w\in A$, $\widetilde{w}\in B$. Then we have 
		
		\noindent$\partial(v,w)\leq\begin{cases}
			i+3, & \text{if } w\in A_i,\ 1\leq i\leq 3,\\
			5, & \text{if } w\in A_4\cup A_9\cup A_{10}^{(1)},\\
			6, & \text{if } w\in A_5\cup A_6,\\
			7, & \text{if } w\in A_7,\\
			7, & \text{if } w \in A_8 ,\ K'_1=\emptyset,\\
			7, & \text{if } w\in A_{10}^{(2)},\ I'_3=\emptyset,\\
			8, & \text{if } w\in A_8,\  K'_1\neq\emptyset\\
			8, & \text{if } w\in A_{10}^{(2)},\  I'_3\neq\emptyset;
		\end{cases}$ $\partial(\widetilde{w},u)\leq\begin{cases}
		i+3, & \text{if } \widetilde{w}\in B_i,\ 1\leq i\leq 3,\\
		5, & \text{if } \widetilde{w}\in B_4\cup B_9\cup B_{10}^{(1)},\\
		6, & \text{if } \widetilde{w}\in B_5\cup B_6,\\
		7, & \text{if } \widetilde{w}\in B_7,\\
		7, & \text{if } \widetilde{w} \in B_8 ,\ L'_1=\emptyset,\\
		7, & \text{if } \widetilde{w}\in B_{10}^{(2)},\ J'_3=\emptyset,\\
		8, & \text{if } \widetilde{w}\in B_8,\  L'_1\neq\emptyset\\
		8, & \text{if } \widetilde{w}\in B_{10}^{(2)},\  J'_3\neq\emptyset.
		\end{cases}$
	\end{claim}
	
	\begin{proof}
		By symmetry, we only prove the case of $w\in A$. For $1\leq i\leq 8$, we have $N(w)\cap I_i\neq\emptyset$ for $w\in A_i$ by the definition of $A_i$, and $I_i\rightarrow A_i$ by Constructions \ref{con1} and \ref{con3}, which implies the result holds for $w\in \cup_{i=1}^8 A_i$ by Claims \ref{claim16} and \ref{claim19}. 
		
		If $w\in A_9$, then $N(w)\cap A_1\neq \emptyset$ by the definition of $A_9$, and thus $\partial(v,w) \leq \partial(v,w')+\partial(w',w)\leq 4+1=5$ for $w'\in N(w)\cap A_1$ by Construction \ref{con1}.
		
		If $w\in A_{10}^{(1)}$, then there exists $P_{w'}=w'wi_2i_3i_4$ for some $w'\in I'_1$ by the definition of $A_{10}^{(1)}$, and thus $i_2\in I'$, $i_3\in I_1$ and $i_4\in J$ by the definition of $I'_1$, which implies $N(i_2)\cap I_1\neq\emptyset$, and $i_4\in J_1$ by $N(i_4)\cap I_1\neq\emptyset$. Therefore, we have $v\rightarrow B\rightarrow i_4\rightarrow i_3\rightarrow i_2\rightarrow w$ by Construction \ref{con1} and (i) of Proposition \ref{prop1.5}, which implies $\partial(v,w)\leq 5$. 
		
		If $w\in A_{10}^{(2)}$, then $N(w)\cap (I'-I'_{14}-I'_{84})\neq \emptyset$ by the definitions of $A_{10}^{(2)},I'_{14}$ and $I'_{84}$, and $I'\rightarrow w$ by Construction \ref{con5}, which implies $\partial(v,w)\leq\partial(v,w')+\partial(w',w)\leq 7+1= 8$ for $w'\in N(w)\cap (I'-I'_{14}-I'_{84})$, with equality only if $I'_3\neq\emptyset$ by $N(w')\cap A_{10}^{(2)}\subseteq N(w')\cap A_{10}\neq\emptyset$, Claims \ref{claim11} and \ref{claim I'_1}. 
	\end{proof}
	
	By Claims \ref{claim9}-\ref{claim20}, the known directed distances from \( v \) and to \( u \) are summarized in Table \ref{Table3}.
	\vspace{3pt}
	\begin{table}[h]
		\centering
		
		\renewcommand{\arraystretch}{1.5}
		\setlength{\tabcolsep}{7pt}
		\begin{tabular}{|l|c|c|c|c|c|c|c|c|c|c|c|c|c|c|c|}
			\hline
			\textbf{for \(w\) in} & \(A\) & \(B\)& $A'$ & $B'$ & \(I\) & \(J\)& $I'$&$J'$  & \(K\) & \(L\)&$K'$&$L'$& $S_{3,3}$ &\(M\)&$M'$  \\ \hline
			\(\partial(w, u) \leq\) & 1 & 8 &2&9& 3 & 7&4 & 7&3 & 5& 4&7&3 & 4&5  \\ \hline
			\(\partial(v, w) \leq\) & 8 & 1 & 9&2&7 & 3 &7& 4&5 & 3 & 7&4&3 &4&5  \\ \hline
		\end{tabular}
		\vspace{3pt}
		\caption{Some known directed distances with $l_G(uv)=g^*(G)\in \{6,7,8\}$.}
		\label{Table3}
	\end{table}
	
	Clearly, Claim \ref{claim21} holds by $\partial(u,v)=1$ and $\partial(x,y)\leq \partial(x,u)+\partial(u,v)+\partial(v,y)$.
	
	\begin{claim}\label{claim21}
		If $\{x,y\}\subseteq   K\cup L\cup S_{3,3}\cup M\cup M'$ or $\{x,y\}\cap \{u,v\}\neq \emptyset$, then $\partial(x,y)\leq 13$.
	\end{claim}
	
	By Claim \ref{claim21}, we only need to consider the case $\{x,y\}\cap (A\cup B\cup A'\cup B'\cup I\cup J\cup  I'\cup J'\cup K'\cup L')\neq \emptyset$ with $\{x,y\}\cap \{u,v\}=\emptyset$.
	
	\begin{claim}\label{claim23}
		Let $x,y\in V(G)$ with $\{x,y\}\cap \{u,v\}=\emptyset$. Then we have\\
		{\rm (i)} if $\{x,y\}\cap (A'\cup B')\neq \emptyset$, then $\partial(x,y)\leq 12$;\\
		{\rm (ii)} if $\{x,y\}\cap (A\cup B)\neq \emptyset$, then $\partial(x,y)\leq 13$;\\
		{\rm (iii)} if $\{x,y\}\cap (I'\cup J')\neq \emptyset$, then $\partial(x,y)\leq 13$; \\
		{\rm (iv)} if $\{x,y\}\cap (I\cup J)\neq \emptyset$, then $\partial(x,y)\leq 13$;\\
		{\rm (v)} if $\{x,y\}\cap (K'\cup L')\neq \emptyset$, then $\partial(x,y)\leq 13$.
	\end{claim}

	{\noindent\textbf{\bf {Proof of (i) of Claim \ref{claim23}.}}} If $x\in A'$ or $y\in B'$, then $\partial(x,y)\leq \partial(x,u)+\partial(u,v)+\partial(v,y)\leq 12$ by Table \ref{Table3}. If $x\in B'$, then $\partial(x,v)\leq 2$ by Claim \ref{claim10}, and hence $\partial(x,y)\leq \partial(x,v)+\partial(v,y)\leq 2+9= 11$ by Table \ref{Table3}. Similarly, if $y\in A'$, then $\partial(u,y)\leq 2$, and hence $\partial(x,y)\leq 11$.{\hfill $\square$\par}
	
	\vspace{0.5\baselineskip}
	{\noindent\textbf{\bf {Proof of (ii) of Claim \ref{claim23}.}}} By (i), we may assume $\{x,y\}\cap (A'\cup B')=\emptyset$. Now we complete the proof by the following three cases.
	
	\vspace{0.1\baselineskip}
	{\noindent \rm \bf Case 1.} $x\in A$, or $y\in B$, or $y\in A$, $x\in I\cup I'\cup K\cup K'\cup S_{3,3}\cup M$, or $x\in B$, $y\in  J\cup J'\cup L\cup L'\cup S_{3,3}\cup M$.
	\vspace{0.1\baselineskip}
	
	By Table \ref{Table3} and $\partial(u,v)=1$, we have $\partial(x,y)\leq \partial(x,u)+\partial(u,v)+\partial(v,y)\leq 13$.
	
	\vspace{0.1\baselineskip}
	{\noindent \rm \bf Case 2.} $y\in A$, $x\in B\cup J\cup J'\cup L\cup L'\cup M'$.
	\vspace{0.1\baselineskip}
		
	\vspace{0.1\baselineskip}
		{\noindent \rm \bf Subcase 1.} $x\in L$ and $y\in A$.
		\vspace{0.1\baselineskip}
		
		Let $x\in L_1$, $y\in A-A_7-A_8-A_{10}$. Then by Claims \ref{claim16} and \ref{claim20}, we have $\partial(x,u)\leq 4$ and $\partial(v,y)\leq 6$, and hence $\partial(x,y)\leq \partial(x,u)+\partial(u,v)+\partial(v,y)\leq 11$.
		
		Let $x\in L_1$ and $y\in A_7\cup A_8\cup A_{10}$. Then we have $N(y)\cap I\subseteq (I_7\cup I_8\cup I')$ by the definitions of $A_7,A_8$ and $A_{10}$, and thus $d(x,y)=4$ by $d(G)=4$. Take $P=xx_1x_2x_3y$ to be a shortest $(x,y)$-path. Then $x_1\in J_1\cup S_{3,3}\cup K_1\cup M$, $x_2\in I_1\cup I_2\cup I_3\cup K'\cup K_2$ and $x_3\in A_1\cup A_2\cup A_3\cup I'\cup I_7\cup I_8$. Now we consider $x_1$ by the following three cases. 
		
		If $x_1\in J_1$, then $x_2\in I_1$ and $x_3\in A_1\cup I'$. For $x_3\in I'$, we have $N(x_3)\cap I_1\neq\emptyset$, which implies $x_2\rightarrow x_3\rightarrow y$ by (i) of Proposition \ref{prop1.5}. Therefore, for $x_3\in A_1\cup I'$, we have $v\rightarrow B\rightarrow x_1\rightarrow x_2\rightarrow x_3\rightarrow y$ by Construction \ref{con1}, which implies $\partial(x,y)\leq \partial(x,u)+\partial(u,v)+\partial(v,y)\leq 4+1+5= 10$ by Claim \ref{claim16}. 
		
		If $x_1\in S_{3,3}\cup K_1$, then $x_2\in I_1\cup I_2\cup I_3\cup K'\cup K_2$ and $x_3\in A_1\cup A_2\cup A_3\cup I'\cup I_7\cup I_8$. For $x_2\in K'$, we have $x_1\in K_1$ and $x_3\in I_7\cup I_8$, then $x_1\rightarrow x_2\rightarrow x_3$ or $\partial(v,x_2)\leq 4$ by $N(x_2)\cap K_1\neq\emptyset$ and (ii) of Proposition \ref{propK'3}, and thus $x\rightarrow x_1\rightarrow x_2\rightarrow x_3\rightarrow y$ or $\partial(v,y)\leq \partial(v,x_2)+\partial(x_2,x_3)+\partial(x_3,y)\leq 4+1+1=6$ by Constructions \ref{con1}, \ref{con3} and \ref{con7}, which implies $\partial(x,y)=4$, or $\partial(x,y)\leq 11$ by Claim \ref{claim16}. For $x_2\notin K'$ and $x_3\notin I'$, we have $x\rightarrow x_1\rightarrow x_2\rightarrow x_3\rightarrow y$  by Constructions \ref{con1} and \ref{con7}, which implies $\partial(x,y)=4$. For $x_2\notin K'$ and $x_3\in I'$, we have $x_2\in \cup_{i=1}^3 I_i$, then $N(x_3)\cap (\cup_{i=1}^3 I_i)\neq\emptyset$, and thus $x_2\rightarrow x_3\rightarrow y$ by (ii) of Proposition \ref{prop1.5}, and thus $x\rightarrow x_1\rightarrow x_2\rightarrow x_3\rightarrow y$ by Construction \ref{con1}, and $\partial(x,y)=4$.
		
		If $x_1\in M$, then $x_2\in K_2$ and $x_3\in I_7$, and thus $x\rightarrow x_1\rightarrow x_2\rightarrow x_3\rightarrow y$ by Construction \ref{con1}, which implies $\partial(x,y)=4$.

		Let $x\in L_2$ and $y\in A_1\cup A_2\cup A_4\cup A_{9}$. Then we have $\partial(x,u)\leq 5$ and $\partial(v,y)\leq 5$ by Claims \ref{claim16} and \ref{claim20}, and thus $\partial(x,y)\leq 11$.	
		
		Let $x\in L_2$ and $y\notin A_1\cup A_2\cup A_4\cup A_{9}$. Then we have $N(x)\cap (S_{3,3}\cup K)=\emptyset$ and $N(y)\cap (I_1\cup I_2)=\emptyset$ by the definitions of $L_2,A_1$ and $A_2$, and thus $d(x,y)=4$. Let $Q=xy_1y_2y_3y$ be a shortest $(x,y)$-path. Then $y_1\in L_1\cup M\cup J_1$, $y_2\in K\cup I_1$ and $y_3\in (I-I_1-I_2)\cup I'\cup A_1$. Now we consider $y_1$ by the following two cases.
		
		If $y_1\in  L_1\cup M$, then $y_2\in K$ and $y_3\in I-I_1-I_2$, and thus $x\rightarrow y_1\rightarrow y_2\rightarrow y_3\rightarrow y$ by Constructions \ref{con1} and \ref{con3}, which implies $\partial(x,y)=4$. 
		
		If $y_1\in J_1$, then $y_2\in I_1$ and $y_3\in  (I-I_1-I_2)\cup I'\cup A_1$, and thus $v\rightarrow B\rightarrow y_1\rightarrow y_2\rightarrow y_3\rightarrow y$ by Construction \ref{con1}, $N(y_3)\cap I_1\neq\emptyset$ and (i) of Proposition \ref{prop1.5}, which implies $\partial(v,y)\leq 5$. Therefore, $\partial(x,y)\leq 11$ by Claim \ref{claim16}.
		
		By the above arguments, we have $\partial(x,y)\leq 11$ for $x\in L$, $y\in A$.
		
		\vspace{0.1\baselineskip}
		{\noindent \rm \bf Subcase 2.} $x\in L'$ and $y\in A$.
		\vspace{0.1\baselineskip}
		
		Let $y\in A_1$. Then we have $\partial(v,y)\leq 4$ by Claim \ref{claim20}, and thus $\partial(x,y)\leq\partial(x,u)+\partial(u,v)+\partial(v,y)\leq 11$ if $x\notin L'_{32}$, and $\partial(x,y)\leq 12$ if $x\in L'_{32}$ by Claims \ref{claim7} and \ref{claim18}. 
		
		
		Let $y\in A_2\cup A_4\cup A_{9}$. Then $\partial(v,y)\leq 5$ by Claim \ref{claim20}, and $L'_1=\emptyset$ by Proposition \ref{prop2}, $A\neq A_1$ and symmetry. Thus $\partial(x,y)\leq 11$ if $x\notin L'_{32}\cup L'_{54}$, and $\partial(x,y)\leq 13$ if $x\in L'_{32}\cup L'_{54}$  by Claims \ref{claim7} and \ref{claim18}. 
		
		Let $y\in A_3\cup A_5\cup A_6\cup A_7\cup A_8\cup A_{10}$. Then $N(y)\cap (I_1\cup I_2)=\emptyset$, and thus $d(x,y)=4$. Take $P=xx_1x_2x_3y$ to be a shortest $(x,y)$-path. Then $x_1\in L_1\cup J_1$, $x_2\in K_1\cup I_1$ and $x_3\in  I_3\cup I'\cup A_1$. Now we consider $x_1$ by the following two cases.
		
		If $x_1\in L_1$, then $x_2\in K_1$, $x_3\in I_3$ and $y\in A_3$, and thus $x\rightarrow L$ or $\partial(x,u)\leq 4$ by (iv) of Proposition \ref{propK'3} and $N(x)\cap L_1\neq\emptyset$. For $x\rightarrow L$, we have $x\rightarrow x_1\rightarrow x_2\rightarrow x_3\rightarrow y$ by Construction \ref{con1}, which implies $\partial(x,y)=4$. For $\partial(x,u)\leq 4$, we have $\partial(x,y)\leq \partial(x,u)+\partial(u,v)+\partial(v,y)\leq 4+1+6= 11$ by Claim \ref{claim20} and $y\in A_3$.
		
		If $x_1\in J_1$, then $x_2\in I_1$, and $x_3\in  I_3\cup I'\cup A_1$, and thus $\partial(x,u)\leq 5$ by $N(x)\cap J_1\neq\emptyset$ and (iii) of Proposition \ref{propK'3}. By Construction \ref{con1}, (i) of Proposition \ref{prop1.5} and $N(x_3)\cap I_1\neq\emptyset$, we have $v\rightarrow B\rightarrow x_1\rightarrow x_2\rightarrow x_3\rightarrow y$, which implies $\partial(v,y)\leq 5$, and thus $\partial(x,y)\leq 11$.

		By the above arguments, we have $\partial(x,y)\leq 11$ if $x\notin L'_{32}\cup L'_{54}$ or $y\notin A_1\cup A_2\cup A_4\cup A_9$, and $\partial(x,y)\leq 13$ otherwise.
				
		\vspace{0.1\baselineskip}
		{\noindent \rm \bf Subcase 3.} $x\in J$ and $y\in A$.
		\vspace{0.1\baselineskip}
		
		Let $x\in J_1\cup J_2\cup J_4$. If $y\in A-A_8-A_{10},x\in J_1\cup J_2\cup J_4$, or $y\in A_8\cup A_{10},x\in J_1$, then $\partial(x,y)\leq  \partial(x,u)+\partial(u,v)+\partial(v,y)\leq  12$ by Claims \ref{claim16} and \ref{claim20}. If $y\in A_8$ and $x\in J_2\cup J_4$, then $\partial(x,y)\leq 13$ by Claims \ref{claim16} and \ref{claim20}. If $y\in A_{10}$ and $x\in J_2\cup J_4$, then $I'_1=I'_3=\emptyset$ by Proposition \ref{prop1.1} and $J\neq J_1$, and thus $\partial(v,y)\leq 7$ by Claim \ref{claim20}, which implies $\partial(x,y)\leq 12$ by Claim \ref{claim16}. 
				
		Let $x\in J_3\cup J_7$. Then $N(x)\cap L\neq\emptyset$ by the definitions of $J_3$ and $J_7$, and $x\rightarrow L$ by Construction \ref{con1}, and thus $\partial(x,y)\leq \partial(x,x')+\partial(x',y)\leq 1+11=12$ for $x'\in N(x)\cap L$ by Subcase 1.
		
		Let $x\in J_5\cup J_6\cup J_8^{(1)}$. Then $N(x)\cap (L\cup L')\neq\emptyset$ by the definitions of $J,J_1$ and $J_2$, and $\partial(x,u)\leq 5$ by Claims \ref{claim16} and \ref{claim19}. If $y\in A_1\cup A_2\cup A_4\cup A_9$, then $\partial(v,y)\leq 5$ by Claim \ref{claim20}, which implies $\partial(x,y)\leq 11$. If $y\notin A_1\cup A_2\cup A_4\cup A_9$, then $\partial(x'',y)\leq 11$ for $x''\in N(x)\cap(L\cup L')$ by Subcases 1-2. Since $x\rightarrow x''$ by symmetry, Constructions \ref{con1} and \ref{con7}, which implies $\partial(x,y)\leq 12$. 
		
		Let $x\in J_8^{(2)}$. Then $N(x)\cap (L'-L'_{32}-L'_{54})\neq\emptyset$ by symmetry, the definitions of $K'_{32}$ and $K'_{54}$, and $x\rightarrow L'-L'_{32}-L'_{54}$ by symmetry and Construction \ref{con7}, which implies $\partial(x,y)\leq 12$ by Subcase 2.
		
		By the above arguments, we have $\partial(x,y)\leq 13$, with equality only if $x\in J_2\cup J_4$ and $y\in A_8$.
		
		\vspace{0.1\baselineskip}
		{\noindent \rm \bf Subcase 4.} $x\in J'$ and $y\in A$.
		\vspace{0.1\baselineskip}
		
		By Claim \ref{claim12}, we have $\partial(x,u)\leq 7$, with equality only if $x\in J'_{14}\cup J'_{84}$, or $x\in J'_3$ with $N(x)\cap J\subseteq J_7\cup J_8$, or $x\in J'_5$ with $N(x)\cap B_{10}=\emptyset$ and $d(x,I)=4$. 
		
		Let $y\in A_1$. Then $\partial(x,y)\leq \partial(x,u)+\partial(u,v)+\partial(v,y)\leq 12$ by Claim \ref{claim20}.	
		
		Now we consider $y\in A-A_1$. By $A\neq A_1$, symmetry and Proposition \ref{prop1.1}, we have $J'_1=J'_3=\emptyset$, and then $J'_{14}=\emptyset$.
		
		Let $y\in  A_2\cup A_4\cup A_9$. Then $\partial(v,y)\leq 5$ by Claim \ref{claim20}. If $x\notin J'_5\cup J'_{84}$, or $x\in J'_5$ with $N(x)\cap B_{10}\neq\emptyset$ or $d(x,I)\leq 3$, then $\partial(x,y)\leq \partial(x,u)+\partial(u,v)+\partial(v,y)\leq 12$. If $x\in J'_{84}$, or $x\in J'_5$ with $N(x)\cap B_{10}=\emptyset$ and $d(x,I)= 4$, then $\partial(x,y)\leq \partial(x,u)+\partial(u,v)+\partial(v,y)\leq 13$. 
		
		Let $y\in A_3\cup A_5\cup A_6\cup A_7\cup A_8\cup A_{10}$. Then we consider the following cases.
		
		If $N(x)\cap J_1\neq\emptyset$, then $x\rightarrow J_1\rightarrow I_1\rightarrow A\rightarrow u$ by (iii) of Proposition \ref{prop1.5} and Construction \ref{con1}, which implies $\partial(x,u)\leq 4$, and thus $\partial(x,y)\leq \partial(x,u)+\partial(u,v)+\partial(v,y)\leq  \begin{cases}
			12, \text{ if } y\notin A_8\cup A_{10}\\
			13, \text{ if } y\in A_8\cup A_{10}		\end{cases}$by Claim \ref{claim20}.
		
	  If $N(x)\cap J_1=\emptyset$, then by $y\in A_3\cup A_5\cup A_6\cup A_7\cup A_8\cup A_{10}$, there exists $y'\in N(y)\cap (I_3\cup I_5\cup I_6\cup I_7\cup I_8\cup I')$, and thus $d(x,y')=4$. 
	  Let $P=xx_1x_2x_3y'$ be a shortest $(x,y')$-path, where $x_1\in  B_1\cup J'\cup (\cup_{i=2}^4 J_i)$, $x_2\in J_1\cup S_{3,3}\cup L_1$ and $x_3\in I_1\cup I_2\cup K_1$. Now we consider $x_1$ by the following two cases.
		
	For $x_1\in B_1\cup J'$, we have $x_2\in J_1$ and $x_3\in I_1$. Thus $d(x,I)=3$, and $v\rightarrow B\rightarrow x_2\rightarrow x_3$ by Construction \ref{con1}. If $y'\notin I'$, then $x_3\rightarrow y'\rightarrow y$ by Construction \ref{con1}; if $y'\in I'$, then $x_3\rightarrow y'\rightarrow y$ by $N(y')\cap I_1\neq\emptyset$ and (i) of Proposition \ref{prop1.5}. Hence $\partial(v,y)\leq \partial(v,x_3)+\partial(x_3,y)\leq 3+2=5$. On the other hand, we have $x\notin J'_{1}\cup J'_3\cup J'_{8}$ by $d(x,I)=3$, then $\partial(x,u)\leq 6$ by Claim \ref{claim12} and $d(x,I)=3$. Therefore, $\partial(x,y)\leq 12$. 
		
	For $x_1\in \cup_{i=2}^4 J_i$, we have $N(x)\cap (\cup_{i=2}^4 J_i)\neq\emptyset$, and thus $x\rightarrow x_1\rightarrow x_2\rightarrow x_3$ by (iv) of Proposition \ref{prop1.5} and Constructions \ref{con1}. On the other hand, if $y'\notin I'$, then $x_3\rightarrow y'\rightarrow y$ by Construction \ref{con1}; if $y'\in I'$, then $x_3\in I_1\cup I_2$, and thus $x_3\rightarrow y'\rightarrow y$ by (ii) of Proposition \ref{prop1.5} and $N(y')\cap (I_1\cup I_2)\neq\emptyset$. Therefore, $\partial(x,y)\leq \partial(x,x_3)+\partial(x_3,y)\leq 3+2=5$.
		
		By the above arguments, we have $\partial(x,y)\leq 13$, with equality only if one of the following three cases holds: (1) $y\in  A_2\cup A_4\cup A_9$ and $x\in J'_{84}$; (2) $y\in  A_2\cup A_4\cup A_9$ and $x\in J'_5$ with $N(x)\cap B_{10}=\emptyset$, $d(x,I)= 4$; (3) $y\in A_8\cup A_{10}$ and $N(x)\cap J_1\neq\emptyset$.
		
		\vspace{0.1\baselineskip}
		{\noindent \rm \bf Subcase 5.} $x\in B$ and $y\in A$.
		\vspace{0.1\baselineskip}
		
		Let $x\in B_i$ with $1\leq i\leq 8$, and $x'\in N(x)\cap J_i$ by the definition of $B_i$. Then $x\rightarrow x'$ by Constructions \ref{con1} and \ref{con3}. If $y\notin A_8$, $\partial(x',y)\leq 12$ by Subcase 3, and hence $\partial(x,y)\leq 13$. If $x\notin B_2\cup B_4$ and $y\in A_8$, then $x'\notin J_2\cup J_4$, and thus $\partial(x',y)\leq 12$ by Subcase 3, which implies $\partial(x,y)\leq 13$. If $x\in B_2\cup B_4$ and $y\in A_8$, then $\partial(x,u)\leq 5$ by Claim \ref{claim20}, and $K'_1=\emptyset$ by Proposition \ref{prop2}. Thus $\partial(v,y)\leq 7$ by Claim \ref{claim20}, and $\partial(x,y)\leq 13$. 
		
		Let $x\in B_9\cup B_{10}$. By Propositions \ref{prop1.1}, \ref{prop2} and the fact $B\neq B_1$, we have $K'_1=I'_1=I'_3=\emptyset$, and thus $\partial(v,y)\leq 7$ by Claim \ref{claim20}. Now we consider the following cases. 
		
		If $x\in B_9\cup B_{10}^{(1)}$, then $\partial(x,u)\leq 5$, and thus $\partial(x,y)\leq 13$. 
		
		If $x\in B_{10}^{(2)}$, then $N(x)\cap (J'-J'_{14}-J'_{84})\neq \emptyset$ by symmetry, the definitions of $I'_{14}$ and $I'_{84}$. Take $x'\in N(x)\cap (J'-J'_{14}-J'_{84})$. Then $x\rightarrow x'$ by Construction \ref{con5} and symmetry. For $y\notin A_8\cup A_{10}$, or $y\in A_8\cup A_{10},\ N(x')\cap J_1=\emptyset$, we have $\partial(x',y)\leq 12$ by Subcase 4 and $x\in N(x')\cap B_{10}\neq\emptyset$, and hence $\partial(x,y)\leq \partial(x,x')+\partial(x',y)\leq 1+12= 13$. For $y\in A_8\cup A_{10}$ and $N(x')\cap J_1\neq\emptyset$, we have $x\rightarrow x'\rightarrow J_1\rightarrow I_1 \rightarrow A\rightarrow u$ by (iii) of Proposition \ref{prop1.5} and Construction \ref{con1}, and thus $\partial(x,u)\leq 5$, which implies $\partial(x,y) \leq 13$.
		
		\vspace{0.1\baselineskip}
		{\noindent \rm \bf Subcase 6.} $x\in M'$ and $y\in A$.
		\vspace{0.1\baselineskip}
		
	By $x\in M'$, we have $\partial(x,u)\leq 5$ by Claim \ref{claim17}. 
	
	If $y\notin A_8\cup A_{10}$, then $\partial(v,y)\leq 7$ by Claim \ref{claim20}, and hence $\partial(x,y)\leq 13$. 
	
	If $y\in A_8\cup A_{10}$, then $d(x,y)=4$. Let $P=xx_1x_2x_3y$ be a shortest $(x,y)$-path. Then $x_1\in S_{3,3}$, $x_2\in I_1\cup I_2$ and $x_3\in A_1\cup A_2\cup I'$. For $x_3\in A_1\cup A_2$, we have $x_3\rightarrow y$ by Construction \ref{con1}, and $\partial(v,x_3)\leq 5$ by Claim \ref{claim20}, and thus $\partial(v,y)\leq 6$, which implies $\partial(x,y)\leq 12$. For $x_3\in I'$, we have $x_2\rightarrow x_3\rightarrow y$ by $N(x_3)\cap (I_1\cup I_2)\neq\emptyset$ and (ii) of Proposition \ref{prop1.5}, and thus $\partial(v,y)\leq \partial(v,x_2)+\partial(x_2,y)\leq 4+2=6$ by Claim \ref{claim16}, which implies $\partial(x,y)\leq 12$.
	
	\vspace{0.1\baselineskip}
	{\noindent \rm \bf Case 3.} $x\in B$, $y\in A\cup I\cup I'\cup K\cup K'\cup M'$.
	\vspace{0.1\baselineskip}
	
	By symmetry, we can show the result similarly to the proof of Case 2, and we omit the details.	{\hfill $\square$\par}
	
		\vspace{0.5\baselineskip}
		{\noindent\textbf{\bf {Proof of (iii) of Claim \ref{claim23}.}}}		
		By (i) and (ii), we may assume $\{x,y\}\cap (A'\cup B'\cup A\cup B)=\emptyset$. Now we complete the proof by the following three cases.
		
		\vspace{0.1\baselineskip}
		{\noindent \rm \bf Case 1.} $x\in I'$, or $y\in J'$, or $x\in J'$, $y\in J\cup K\cup L\cup L'\cup S_{3,3}\cup M\cup M'$, or $y\in I'$, $x\in I\cup K\cup L\cup K'\cup S_{3,3}\cup M\cup M'$.
		\vspace{0.1\baselineskip}

		 By Table \ref{Table3}, we have $\partial(x,y)\leq \partial(x,u)+\partial(u,v)+\partial(v,y)\leq  13$.
		 
		 \vspace{0.1\baselineskip}
		 {\noindent \rm \bf Case 2.} $x\in J'$, $y\in I\cup I'\cup K'$.
		 \vspace{0.1\baselineskip}

		\vspace{0.1\baselineskip}
		{\noindent \rm \bf Subcase 1.} $x\in J'$ and $y\in K'$.
		\vspace{0.1\baselineskip}
	
		If $N(x)\cap J_1\neq \emptyset$, then $x\rightarrow J_1\rightarrow I_1\rightarrow A\rightarrow u$ by Construction \ref{con1}, (iii) of Proposition \ref{prop1.5}, and hence $\partial(x,u)\leq 4$, which implies $\partial(x,y)\leq 12$ by Table \ref{Table3}.
		
		If $N(x)\cap J_1=\emptyset$, then $d(x,y)=4$. Let $P=xx_1x_2x_3y$ be a shortest $(x,y)$-path. Then $x_1\in B_1\cup J'\cup J_4\cup J_2\cup J_3$, $x_2\in J_1\cup S_{3,3}\cup L_1$ and $x_3\in I_1\cup I_2\cup K_1$. For $x_3\in I_1\cup K_1$, we have $\partial(v,y)\leq 5$ by $N(y)\cap (I_1\cup K_1)\neq\emptyset$, (i) and (ii) of Proposition \ref{propK'3}, which implies $\partial(x,y)\leq 13$ by Table \ref{Table3}. For $x_3\in I_2$, we have $x_2\in S_{3,3}$ and $x_1\in J_2$, and thus $x\rightarrow x_1\rightarrow x_2\rightarrow x_3\rightarrow A\rightarrow u$ by (iv) of Proposition \ref{prop1.5}, $N(x)\cap J_2\neq\emptyset$ and Constructions \ref{con1}, which implies $\partial(x,u)\leq 5$ and $\partial(x,y)\leq 13$ by Table \ref{Table3}.
		
		\vspace{0.1\baselineskip}
		{\noindent \rm \bf Subcase 2.} $x\in J'$ and $y\in I$.
		\vspace{0.1\baselineskip}
		
		If $y\notin I_7\cup I_8^{(2)}$, then $\partial(v,y)\leq 5$ by Claims \ref{claim16} and \ref{claim19}, and hence  $\partial(x,y)\leq 13$ by Table \ref{Table3}. If $y\in I_7\cup I_8^{(2)}$ and $N(x)\cap J_1\neq\emptyset$, then $x\rightarrow J_1\rightarrow I_1\rightarrow A\rightarrow u$ by (iii) of Proposition \ref{prop1.5} and Construction \ref{con1}, and thus $\partial(x,u)\leq 4$, which implies $\partial(x,y)\leq 12$ by Table \ref{Table3}. If $y\in I_7\cup I_8^{(2)}$ and $N(x)\cap J_1=\emptyset$, then $d(x,y)\geq 5$, a contradiction with $d(G)=4$.
		
		\vspace{0.1\baselineskip}
		{\noindent \rm \bf Subcase 3.} $x\in J'$ and $y\in I'$.
		\vspace{0.1\baselineskip}
		
		If $N(x)\cap J_1\neq\emptyset$, then $x\rightarrow J_1\rightarrow I_1\rightarrow A\rightarrow u$ by (iii) of Proposition \ref{prop1.5} and Construction \ref{con1}, and thus $\partial(x,u)\leq 4$, which implies $\partial(x,y)\leq 12$. 
		
		Now, we assume $N(x)\cap J_1=\emptyset$. Then $d(x,y)=4$. Let $Q=xy_1y_2y_3y$ be a shortest $(x,y)$-path. Thus $y_1\in B_1\cup J'\cup  (\cup_{i=2}^4 J_i)$, $y_2\in J_1\cup S_{3,3}$ and $y_3\in I_1\cup I_2$. By $N(y)\cap (I_1\cup I_2)\neq\emptyset$ and (ii) of Proposition \ref{prop1.5}, we have $y_3\rightarrow y$, then $\partial(v,y)\leq \partial(v,y_3)+\partial(y_3,y)\leq 4+1=5$ by Claim \ref{claim16}, and thus $\partial(x,y)\leq 13$ by Table \ref{Table3}.
		
		\vspace{0.1\baselineskip}
		{\noindent \rm \bf Case 3.} $y\in I'$, $x\in J\cup J'\cup L'$.
		\vspace{0.1\baselineskip}
		
		By symmetry, we can show the result by the similar proof of Case 2, and we omit the details.
		{\hfill $\square$\par}
	
		\vspace{0.5\baselineskip}
		{\noindent\textbf{\bf {Proof of (iv) of Claim \ref{claim23}.}}} By (i)-(iii), we may assume $\{x,y\}\cap (A'\cup B'\cup A\cup B\cup I'\cup J')=\emptyset$. Now we complete the proof by the following three cases.
		
		\vspace{0.1\baselineskip}
		{\noindent \rm \bf Case 1.} $x\in I$, or $y\in J$, or $x\in J$, $y\notin I\cup K'$, or $y\in I$, $x\notin J\cup L'$.
		\vspace{0.1\baselineskip}
		
		By Table \ref{Table3}, we have $\partial(x,y)\leq \partial(x,u)+\partial(u,v)+\partial(v,y)\leq  13$.
		
		\vspace{0.1\baselineskip}
		{\noindent \rm \bf Case 2.} $x\in J$, $y\in I\cup K'$.
		\vspace{0.1\baselineskip}

		\vspace{0.1\baselineskip}
		{\noindent \rm \bf Subcase 1.} $x\in J$ and $y\in I$.
		\vspace{0.1\baselineskip}
		
		If $x\notin J_7\cup J_8^{(2)}$ or $y\notin I_7\cup I_8^{(2)}$, then $\partial(x,u)\leq 5$ or $\partial(v,y)\leq 5$ by Claims \ref{claim16} and \ref{claim19}, and hence $\partial(x,y)\leq 13$ by Table \ref{Table3}. 
		
		Now we assume $x\in J_7\cup J_8^{(2)}$ and $y\in I_7\cup I_8^{(2)}$. Then $d(x,y)=4$. Let $P=xx_1x_2x_3y$ be a shortest $(x,y)$-path. Then $x_1\in L_2$, $x_2\in M$ and $x_3\in K_2$. By Construction \ref{con1}, $x\rightarrow x_1\rightarrow x_2\rightarrow x_3\rightarrow y$, which implies $\partial(x,y)= 4$.
		
		\vspace{0.1\baselineskip}
		{\noindent \rm \bf Subcase 2.} $x\in J$ and $y\in K'$.
		\vspace{0.1\baselineskip}
		
		Let $x\in (\cup_{i=1}^6 J_i)\cup J_8^{(1)}$. Then $\partial(x,u)\leq 5$ by Claims \ref{claim16} and \ref{claim19}, and thus $\partial(x,y)\leq 13$ by Table \ref{Table3}. 
		
		Let $x\in J_7\cup J_8^{(2)}$. Then we consider $N(y)$.	If $N(y)\cap I_1\neq\emptyset$, then $\partial(v,y)\leq 5$ by (i) of Proposition \ref{propK'3}, which implies $\partial(x,y)\leq 13$ by Table \ref{Table3}. If $N(y)\cap I_1=\emptyset$, then $d(x,y)=4$. Let $P=xx_1x_2x_3y$ be a shortest $(x,y)$-path. Then $x_1\in J_3\cup L_2\cup L'$, $x_2\in L_1\cup M$ and $x_3\in K$. For $x_2\in L_1$, we have $x_3\in K_1$, then $N(y)\cap K_1\neq\emptyset$, and thus $\partial(v,y)\leq 5$ by (ii) of Proposition \ref{propK'3}, which implies $\partial(x,y)\leq 13$ by Table \ref{Table3}. For $x_2\in M$, we have $x_1\in L_2$, $x_3\in K$ and $N(y)\cap K\neq\emptyset$, then $x\in J_7$ by $N(x)\cap L_2\neq\emptyset$, and $y\notin K'_{32}$ by the definition of $K'_{32}$, and thus $\partial(x,u)\leq 6$ and $\partial(v,y)\leq 6$ by Claims \ref{claim16}, \ref{claim6} and \ref{claim18}, which implies $\partial(x,y)\leq 13$.	
		
		\vspace{0.1\baselineskip}
		{\noindent \rm \bf Case 3.} $y\in I$, $x\in J\cup L'$.
		\vspace{0.1\baselineskip}
		
		By symmetry, we can show the result by the similar proof of Case 2, and we omit the details.
		{\hfill $\square$\par}
	
		\vspace{0.5\baselineskip}
	{\noindent\textbf{\bf {Proof of (v) of Claim \ref{claim23}.}}}
	By (i)-(iv), we may assume $\{x,y\}\cap (A'\cup B'\cup A\cup B\cup I\cup J\cup  I'\cup J')=\emptyset$. By Table \ref{Table3}, we have $\partial(x,y)\leq 13$ if $x\in K'$, or $y\in L'$, or $x\in L'$, $y\notin K'$, or $y\in K'$, $x\notin  L'$. Then we only show the case $x\in L'$, $y\in K'$. 
	
	Let $x\notin L'_{32}$ and $y\notin K'_{32}$. Then $\partial(x,u)\leq 6$ and $\partial(v,y)\leq 6$ by Claims \ref{claim6} and \ref{claim18}, and hence $\partial(x,y)\leq 13$. Now we assume $x\in L'_{32}$ or $y\in K'_{32}$. By symmetry, we only consider $y\in K'_{32}$. 
	
	If $\partial(v,y)\leq 5$, then $\partial(x,y)\leq 13$ by Table \ref{Table3}. 
	
	If $\partial(v,y)>5$, then $\partial(v,y)\in \{6,7\}$, and $y$ is not isolated in $G[K']$ with $N(y)\cap K'_2=\emptyset$ by Claim \ref{claim18}.	By the definitions of $K'_3$ and $K'_{32}$, we have $N(y)\cap K'_1\neq\emptyset$ and $N(y')\cap (I-I_8^{(1)})=\emptyset$ for any $y'\in N(y)\cap K'_1$. Then $N(y')\cap I\subseteq I_8^{(1)}$. Now we show $N(y')\cap K_1=\emptyset$. Otherwise, there exists a shortest $(y',B)$-path $Q_{y'}=y'i_1i_2i_3i_4$, where $i_1\in K_1,i_2\in S_{3,3}\cup L_1,i_3\in \cup_{i=1}^3 J_i,i_4\in B$ by $d(y',B)=4$ and the definition of $K_1$, then $y'\in K'_2$ by the definition of $K'_2$, a contradiction. Therefore, we have $d(x,y')=4$. Let $P=xx_1x_2x_3y'$ be a shortest $(x,y')$ path. Then $x_1\in J_1\cup L$, $x_2\in  I_1\cup M\cup K_1$ and $x_3\in K'\cup K_2$. If $x\in L'_{31}$, then $\partial(x,u)\leq 5$ by Claim \ref{claim18}, and hence $\partial(x,y)\leq 13$ by Table \ref{Table3}. If $x\notin L'_{31}$, then $x\rightarrow L$ by symmetry and (iii) of Construction \ref{con13}.
	
	For $x_1\in J_1$, we have $\partial(x,u)\leq 5$ by $N(x)\cap J_1\neq\emptyset$ and (iii) of Proposition \ref{propK'3}, which implies $\partial(x,y)\leq 13$.
	
	For $x_1\in L$, we have $x_2\in M\cup K_1$ and $x_3\in K_2\cup K'$. If $x_3\in K'$, then $x_2\in K_1$, and thus $x_3\in K'_2$ by $d(x_3,B)=4$. Otherwise, we have $d(x_3,B)=3$, and let $x_3w_1w_2w_3$ be a shortest $(x_3,B)$-path, then $y'x_3w_1w_2w_3$ is a shortest $(y',B)$-path by $d(y',B)=4$, a contradiction with the definition of $K'_1$ and $y'\in K'_1$. Therefore, $x_3\in K'_2$ and $x\rightarrow x_1\rightarrow x_2\rightarrow x_3\rightarrow y'\rightarrow y$ by Constructions \ref{con1}, \ref{con7} and (ii) of Construction \ref{con13}, which implies $\partial(x,y)\leq 5$. If $x_3\in K_2$, then $x\rightarrow x_1\rightarrow x_2\rightarrow x_3\rightarrow y' \rightarrow y$ by Constructions \ref{con1}, \ref{con7} and (ii) of Construction \ref{con13}, which implies $\partial(x,y)\leq 5$. 
	{\hfill $\square$\par}
	\vspace{0.3\baselineskip}
	
	By combining Claim~\ref{claim21} and parts (i)-(v) of Claim~\ref{claim23}, 
	we obtain $\overrightarrow{{diam}}(G)\le d(\overrightarrow{G})\le 13$, 
	thereby completing the proof of part~(ii) of Theorem~\ref{t3}. 	{\hfill $\blacksquare$\par}
	
	\vspace{0.5\baselineskip}
	{\noindent\textbf{\bf {Proof of Theorem \ref{t5}.}}}
	By Theorems \ref{t2}-\ref{t3} and the definition of $F(d,g^*)$, we have $F(4,2)=4$, $F(4,3)\leq 12$, $F(4,9)\leq 12$, and $F(4,g^*)\leq 13$ for $g^*\in\{6,7,8\}$. Now we show $F(4,9)\geq 12$.
	
	In the proof of \cite[Proposition 2.1]{PK}, Kwok et al. constructed a graph $H_0$ (see Figure \ref{fig2})	 with $d(H_0)=4$ and $g^*(H_0)=9$ such that its oriented diameter is at least $12$, that is, $ \overrightarrow{diam}(H_0)\geq 12$. 
	\begin{figure}[h]
		\centering
		\includegraphics[scale=0.2]{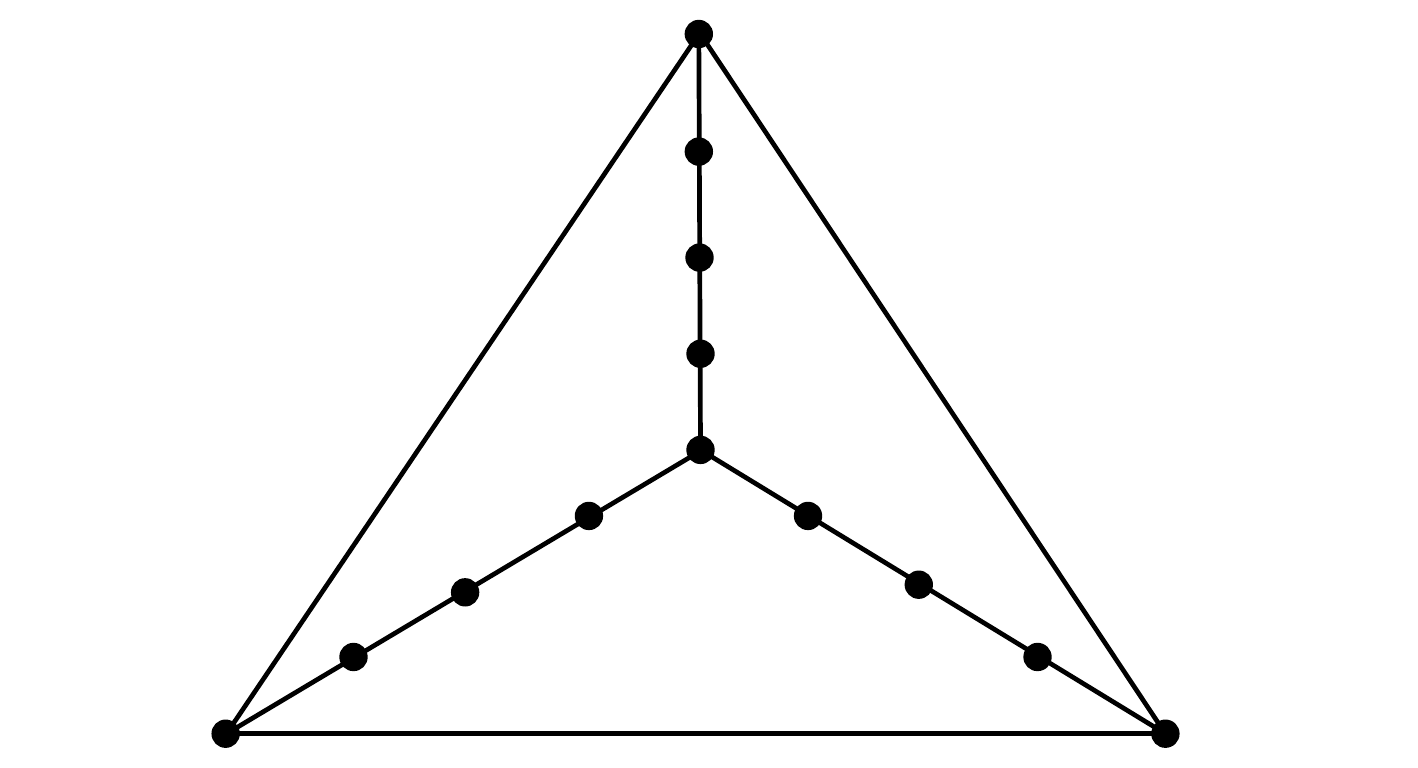}
		\caption{$H_0$ is the subdivision of $K_4$.}\label{fig2}
	\end{figure}
	Then we obtain $F(4,9)\geq 12$, and thus $F(4,9)=12$. {\hfill $\blacksquare$\par}
	
\section{Further Research Problems}
\hspace{1.5em}
To determine the exact value of \(f(4)\), the following problems need to be considered. In the proof of (ii) of Theorem \ref{t3}, we find it difficult to construct an orientation $\overrightarrow{G}$ of \( G \) such that \( d(\overrightarrow{G}) \leq 12 \). Thus we pose the following problem.

\begin{problem}
	Construct a bridgeless graphs $G$ with diameter $d(G)=4$ such that \\$\overrightarrow{diam}(G)\geq 13$.
\end{problem}

    By Theorem \ref{t5}, we have $F(4,g^*)\leq 13$ for $g^*\in\{2,3,6,7,8,9\}$, and then we see $f(4)\leq\max\{13,F(4,4),F(4,5)\}$ by $f(4)=\max\{F(4,g^*)\mid 2\leq g^*\leq 9\}$. In \cite{BD}, Babu et al. gave $f(4)\leq 21$, which implies $\max\{F(4,4),F(4,5)\}\leq 21$. Then we can propose a problem as follows.
\begin{problem}
	Provide improved upper bounds for $F(4,4)$ and $F(4,5)$.
\end{problem}	
		
\section*{Funding}
	\hspace{1.5em}This work is supported by the National Natural Science Foundation of China (Grant Nos. 12371347, 12271337).

\section*{Acknowledgements}
\hspace{1.5em}We sincerely thank the anonymous reviewers for their helpful comments and suggestions.

\end{document}